\documentclass[final]{siamart1116}

\allowdisplaybreaks
\usepackage{pgf}
\usepackage{lipsum}
\usepackage{amsfonts}
\usepackage{graphicx}
\usepackage{epstopdf}
\usepackage{algorithmic}
\ifpdf
  \DeclareGraphicsExtensions{.eps,.pdf,.png,.jpg}
\else
  \DeclareGraphicsExtensions{.eps}
\fi
\usepackage{color}
\usepackage{epsfig}
\usepackage{psfrag}
\usepackage{graphicx}
\usepackage{textcomp}
\usepackage{pgf}   
\usepackage{upgreek} 
\usepackage{mathrsfs,cite}

\numberwithin{theorem}{section}

\newcommand{\WE}[1]{{\color{black}#1}}
\newcommand{\we}[1]{{\color{black}#1}}
\newcommand{\REF}[1]{{\color{black}#1}}
\newcommand{\REFF}[1]{{\color{black}#1}}
\newcommand{\reff}[1]{{\color{black}#1}}

\newcommand{\TheTitle}{Thermodynamics of 
elastoplastic porous
rocks at large strains \REFF{towards} earthquake
\REFF{modeling}
} 
\newcommand{\TheAuthors}{T. Roub\'i\v cek and U. Stefanelli}

\newtheorem{remark}[theorem]{Remark}

\newcommand{\ITEM}[2]{\parbox[t]{.05\textwidth}{{\rm #1}}\hfill\parbox[t]{.95\textwidth}{#2}\vspace*{.8mm}}

        \newcommand{\COMMENT}[1]{}

\newcommand{\EE}{\color{black}} 

\newcommand{\R}{\mathbb{R}}

\newcommand{\bbI}{\mathbb{I}}
\newcommand\DT[1]{\mathchoice
                 {{\buildrel{\hspace*{.1em}\text{\LARGE.}}\over{#1}}}
                 {{\buildrel{\hspace*{.1em}\text{\Large.}}\over{#1}}}
                 {{\buildrel{\hspace*{.1em}\text{\large.}}\over{#1}}}
                 {{\buildrel{\hspace*{.1em}\text{\large.}}\over{#1}}}}
\newcommand\DDT[1]{\mathchoice
   {{\buildrel{\hspace*{.13em}\text{\LARGE.\hspace*{-.13em}.}}\over{#1}}}
   {{\buildrel{\hspace*{.1em}\text{\Large.\hspace*{-.1em}.}}\over{#1}}}
   {{\buildrel{\hspace*{.1em}\text{\large.\hspace*{-.1em}.}}\over{#1}}}
   {{\buildrel{\hspace*{.1em}\text{\large.\hspace*{-.1em}.}}\over{#1}}}}

\newcommand{\linesunder}[3]{\LSU{\begin{array}[t]{c}\underbrace{#1}\vspace*{.5em}\end{array}}{\mbox{\footnotesize\rm #2}}{\mbox{\footnotesize\rm#3}}}

\newcommand{\LSU}[3]{\begin{array}[t]{c}#1\vspace*{-1em}\\_{#2}\vspace*{-.3em}\\_{#3}\end{array}}
\newcommand{\Vdots}{\vdots}
\renewcommand{\d}{{\rm d}}
\newcommand\Frakm{\text{\large$\mathbb{M}$}}
\newcommand\Frakk{\text{\large$\mathbb{K}$}}

\newcommand\FT{\psi_{_{\rm T}}}
\newcommand\FM{\psi_{_{\rm M}}}
\newcommand\widehatFM{\widehat\psi_{_{\rm M}}}
\newcommand\PsiT{\Psi_{_{\rm T}}}
\newcommand\PsiM{\Psi_{_{\rm M}}}

\newcommand\PP{\varPi}
\newcommand\PR{P}
\newcommand\Cof{\mathrm{Cof}}

\newcommand{\divS}{\mathrm{div}_{\scriptscriptstyle\textrm{\hspace*{-.1em}S}}^{}}
\newcommand{\tauR}{\tau_{\scriptscriptstyle\textrm{\hspace*{-.3em}\rm rel}}^{}}
\newcommand{\UU}[2]{\begin{array}[b]{c}_{\mbox{\footnotesize{#2}}}\vspace*{-.5mm}\\#1\end{array}}
\newcommand{\por}{\mu}
\newcommand{\GG}{G}
\newcommand{\PORO}{\sigma} 

\headers{Thermodynamics of porous rocks at large strains}{\TheAuthors}

\title{{\TheTitle}\thanks{Submitted to the editors DATE.
\funding{The authors acknowledge the hospitality and the support of the
Erwin Schr\"odinger Institute of the University of Vienna, where most
of this research has been performed. T.R.\ acknowledges also the support of 
CSF (Czech Science Foundation) project 16-03823S 
and 17-04301S 
and also by the Austrian-Czech projects 16-34894L (FWF/CSF) and 7AMB16AT015 
(FWF/MSMT CR) as well as through the institutional support RVO:\,61388998 
(\v CR). U.S.\ acknowledges the support of the Austrian Science Fund (FWF)
 projects  F\,65,  P\,27052, and I\,2375 and of the Vienna Science and 
Technology Fund (WWTF)project MA14-009.}}}

\author{
  Tom\' a\v s Roub\'i\v{c}ek\thanks{Mathematical Institute, Charles University,
Sokolovsk\'a 83, CZ-186~75~Praha~8,  Czech Republic and 
Institute of Thermomechanics, Czech Academy of Sciences,
Dolej\v skova 5, CZ-182~00~Praha~8, Czech Republic
   (\email{tomas.roubicek@mff.cuni.cz}).}
  \and
  Ulisse Stefanelli\thanks{Faculty of Mathematics, University of
    Vienna, Oskar-Morgenstern-Platz 1, 1090 Vienna, Austria
 and Istituto di Matematica Applicata e Tecnologie Informatiche
  E. Magenes - CNR, v. Ferrata 1, 27100 Pavia, Italy
  (\email{ulisse.stefanelli@univie.ac.at}).}
}

\usepackage{amsopn}

\ifpdf
\hypersetup{
  pdftitle={\TheTitle},
  pdfauthor={\TheAuthors}
}
\fi

\begin{document}

\maketitle

\begin{abstract}
A mathematical model for an elastoplastic porous continuum subject to large 
strains in combination with reversible damage (aging), evolving porosity, 
water and heat transfer is advanced. The inelastic response is modeled within 
the frame of plasticity for 
nonsimple materials. Water and heat diffuse through the continuum by a 
generalized Fick-Darcy law in the context of viscous Cahn-Hilliard dynamics and 
by Fourier law, respectively. This coupling of phenomena is paramount to the 
description of lithospheric faults, which experience ruptures (tectonic 
earthquakes) originating seismic waves and flash heating. In this regard,
we combine in a thermodynamic consistent way the assumptions of having a small 
\WE{Green-Lagrange} elastic strain and \WE{nearly isochoric plastification} 
with the very large displacements generated by fault shearing. The model is 
amenable to a rigorous mathematical analysis. Existence of suitably defined 
weak solutions and a convergence result for Galerkin approximations is proved.

\end{abstract}



\baselineskip=13pt

\section{Introduction} The global movement of 
tectonic plates in the upper lithospheric mantle originates
{\it tectonic earthquakes}.  These occur on {\it fault zones},
which are relatively 
localized regions of partly damaged rocks with weakened elastic properties and 
weakened shear-stress resistance.  Tectonic earthquakes  are very 
complex thermomechanical events,  often having a   devastating
 societal and economical impact. Correspondingly, they are
intensively investigated by the   geophysical community  under  
various aspects, ranging  from  observation,  to experiments and modeling. 
Despite the extensive information available, the possibility of offering 
reliable prediction of future events seems to be still out of reach 
\cite{Cocc15AT}. 

The dynamics of every lithospheric fault is to some extent unique and 
is often part of a complex and  mutually interacting system. Some typical 
fault  geometry,  although necessarily very idealized with respect to real 
systems but nevertheless used in numerical simulations 
\cite{LyaBeZ09EGMP,LyHaBZ11NLVE}, is depicted in Figure~\ref{fig-geom}. 
\begin{figure}[th]
\begin{center}
\psfrag{x1}{\footnotesize $x_1$}
\psfrag{x2}{\footnotesize $x_2$}
\psfrag{W}{\footnotesize $\Omega$}
\psfrag{y(W)}{\footnotesize $y(\Omega)$}
\psfrag{y}{\footnotesize $y$}
\psfrag{eps1}{\footnotesize $\epsilon_1^{}$}
\psfrag{eps2}{\footnotesize $\epsilon_2^{}$}
\psfrag{compact rock}{\scriptsize {\bf compact rock}}
\psfrag{damaged rock}{\scriptsize {\bf damaged rock}}
\psfrag{core zone}{\scriptsize {\bf core zone}}
\psfrag{g}{\footnotesize $g$}
\psfrag{reference configuration}{\footnotesize {\bf reference (material) configuration}}
\psfrag{actual configuration}{\footnotesize {\bf actual (space) configuration}}
\includegraphics[width=.9\textwidth]{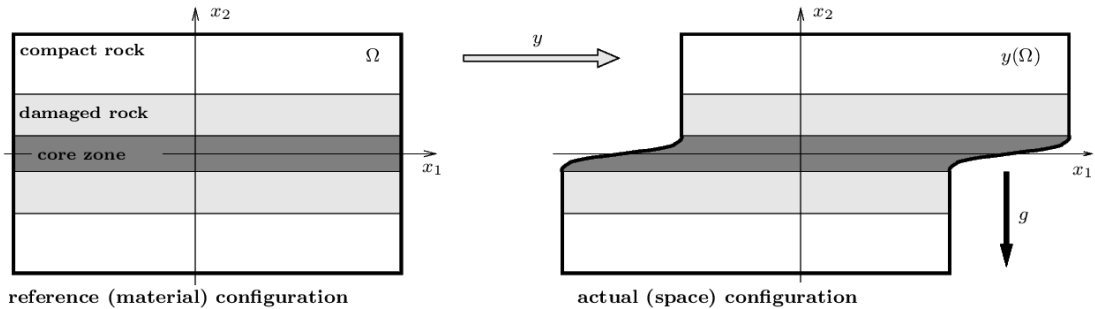}
\end{center}
\caption{\sl Schematic geometry of the fault zone  in the  material reference 
Lagrangian configuration (left)
and the actual space Eulerian configuration deformed by a mapping 
$y\in H^2(\Omega;\R^d)$ (right).  The  (possibly inhomogeneous) gravity force 
$g$ is prescribed naturally in the space configuration.}
\label{fig-geom}
\end{figure} 
As effect of a deformation, the fault zone is sheared and damage is 
accumulated in a relatively narrow region ({\it damage region}) with a width 
of tens to hundreds of meters. Strains are mainly concentrated in the even 
narrower {\it core zone}, whose width ranges typically from centimeters to 
meters. The core can accommodate slips of the order of kilometers within 
millions of years \cite{BeZiSam03CFZ}. This distinguished, multi-scale
nature of fault dynamics can be tackled at different levels, ranging from the 
continuum-Euclidean (here meaning {\it continuum mechanical or 
thermomechanical}) description \REFF{of the} faults, to the granular 
description of fault structures and deformation fields, to the fractal nature 
of the fault network \cite{BeZiSam03CFZ}. 

The focus \REFF{of this contribution is on a} description of seismic 
processes on faults by advancing a thermodynamically consistent model of 
{\it large-strain dynamics} in poroelastic rocks in terms of deformation, 
temperature, plastic and damage dynamics, \REFF{and} water content and 
porosity evolution. The model is detailed in Section \ref{sec-model} below 
and includes in particular the following main features:\\
\ITEM{$\ \circ$}{{\it Large-strain elastoplastic} response combined with fast {\it damage}
 ({\it rupture}) and {\it emission of seismic waves} and their
propagation.}
\ITEM{$\ \circ$}{{\it Flash heating:} intense heat production during strong earthquakes 
influencing damage and sliding resistance, as in the  Dieterich-Ruina 
friction model \cite{Diet79MRF,Ruin83SISVFL}, see also
\cite{Roub14NRSD}.}
\ITEM{$\ \circ$}{Modelization of saturated {\it water flow} and its influence on
material response and eventually on earthquake dynamics \cite{Biot41GTTS}.}
\ITEM{$\ \circ$}{{\it Healing} (also called {\it aging}) and gradual conversion of 
elastic strain to permanent inelastic deformation during 
{\it long-lasting creep} and material degradation.}

The evolution in time of poroelastic rocks in upper lithospheric mantle 
are described as originating by the balance of energy-storage and dissipation 
mechanisms. In particular, we focus on a general form of free energy. 
This loses convexity upon damaging,  as proposed in 
\cite{LRWS97NEBD} and later used in several articles as, e.g.\ 
\cite{HaLyAg05DRDC,LyaHam07DEFF,LyHaBZ11NLVE}.  Such free energy is
augmented by nonlocal energetics in form of a gradient damage and plastic 
theory \cite{LyHaBZ11NLVE,LyaBeZ14CDBF} \WE{and a strain gradient} 
in the frame of so-called 
{\it 2nd-grade nonsimple materials} \cite{FriGur06TBBC,Podi02CISM,Silh88PTNB}
(proposed under the name \REF{``}materials with \REFF{couple stresses''} by 
R.A.\,Toupin 
\cite{Toup62EMCS}). \WE{Such materials are also known as  {\it weakly
    nonlocal}. Nonlocal-material concepts have the capacity to be fitted 
with dispersion of elastic waves in general, cf.\ \cite{Jira04NTCM} for a 
thorough discussion.}  
This effectively entails the control the scale of
the damage and core regions. Eventually, the distinguished variational 
structure for the model allows allow for a comprehensive mathematical 
treatment, including existence of suitably-defined weak solutions, and 
convergence of a Galerkin approximation combined with a regularization.

With respect to previous geophysical modeling 
\cite{LyBZAg97DDFF,HaLyAg05DRDC,LyaHam07DEFF,LyHaBZ11NLVE,LyaBeZ14CDBF} 
the novelty of our contribution is threefold. \\
\ITEM{(i)}{
Our model deals with large strains in a thermodynamically consistent  way. 
This seems to move substantially forward with respect to the current
literature, where
%
%
description are either restricted to small strains
\REFF{or combines small elastic strains with large displacements (but
not completely consistently, as noticed in \cite{Roub17GMHF}).}
The present model possesses a clear \REFF{global} energetics 
\REFF{which can serve for a-priori estimates and rigorous analysis}.
}
\ITEM{(ii)}{
By taking advantage of the variational nature of the model we are in the 
position of presenting a full coupling of effects. Mechanical and thermal 
evolution are consistently coupled with  
damage, porosity, and water content dynamics via the specification
of the energy and dissipation potentials. Constitutive relations
are directly defined in terms of variations of these potentials and
combined with conservation of momenta and energy and internal dynamics.}
\ITEM{(iii)}{
We derive a sound approximation and existence theory. In particular, we 
present a stable and convergent Galerkin-approximation scheme. This is 
unprecedented, to our knowledge, for such a comprehensive model at finite 
strains. One has to remark that the implementation of large-strain models is 
often computationally challenging in comparison to the small-strain
models. Nevertheless, actual computations based on an updated Lagrangian 
scheme (see \cite{Cund89NELF}, for instance) may 
\REF{combine with }
the present model toward simulations.}
\REFF{One has to mention that poroelastic models at large strains
  have already been considered from the engineering viewpoint, cf.\ 
\cite{Boer05TCMP,CheAna10CTFP,DuSoFi10TSMF,GovSim93CSD2,HZZS08TCDL},
where nevertheless no rigorous analysis is addressed.}

The plan of the paper is as follows. In Section~\ref{sec-model} we formalize 
the model. In particular, we specify the form of the total energy and of 
the dissipation. This 
\REF{leads} to the formulation of an evolution system of 
partial differential equations and inclusions. The thermodynamic consistency 
of the model and various possible modifications are also discussed. 
Section~\ref{sec-anal} presents a variational notion of solution as well as 
the main analytical statements. The existence proof via Galerkin's 
approximation is then detailed in Section~\ref{sec-proof}, with most
technical mathematical arguments being related to the 
thermoviscoplasticity and just refer to \cite{RouSte??TELS}. 
Finally, Sect.\,\ref{sec-concl} discusses some improvements of the model
and related analytical complications.


\section{Thermodynamical modeling}\label{sec-model}
We devote this section to present our general model
for damageable poroelastic continua with water and heat transfer. This
is formulated in Lagrangian  coordinates with $\Omega\subset\R^d$ ($d=2$ or  
$3$) being a bounded smooth reference (fixed) configuration. 
 The  variables of the model are
\begin{align*}
&y:\Omega \to \R^d \quad&&\text{deformation,}&&\\ 
&\PP:\Omega \to
\WE{{\rm GL}^+(d)}
&&
\text{plastic part of the inelastic
strain,}\\
&\alpha:\Omega \to \R\quad&&\text{damage  descriptor  (also called aging),}\\
&\phi: \Omega \to \R\quad&&\text{porosity (effectively 
the volumetric part of the inelastic strain)},\\ 
&\zeta: \Omega \to [0,1]\quad&&\text{volume fraction of water,}\\ 
& \theta:\Omega \to (0,\infty)\quad&&\text{absolute 
temperature,} 
\end{align*}
\WE{where ${\rm GL}^+(d)$ denotes the general linear group of matrices
from $\R^{d\times d}$ with positive determinant.}
%
\REFF{We emphasize that, although we have in mind saturated flows, we 
distinguish between water content and porosity. Indeed, compared to
rocks, water is substantially 
compressible. Note that in the standard Biot model $\zeta\sim\phi$ can be
achieved only asymptotically if $\beta=0$ and 
$m\to\infty$ in \eqref{ansatz} below. Beside this interpretation, one can
also think about a double-porosity model where the diffusant is
transfered only by one system of pores.}

For convenience,  we anticipate in Table~\ref{Tab_Notation}
the main notation, to be introduced in this section\REF{; for 
basic notions from continuum (thermo/poro)mechanics at large strains we
refer e.g.\ to the monographs 
\cite{Antm95NPE,Bear18MPFT,Boer05TCMP,Ciar88ME1,Cous04P,GroMaz84NET,Gurt81TFE,KruRou18MMCM}}. \REFF{In particular, note that $P$ is the rate of plastic 
strain in the intermediate configuration \cite{Miel02FELG}.
Here we should also note that we follow the terminological 
conventions in mechanics, which differs from what is used in engineering. 
In particular, we call {\it stored energy} all temperature-independent terms in the free energy. 
}

\begin{table}[ht]
\centering
\fbox{
\begin{minipage}[t]{0.43\linewidth}
\small

$\Omega$ reference  configuration, 

$\Gamma$ boundary of $\Omega$,

$I:=[0,T]$ fixed time interval,

$Q:=I{\times}\Omega$,

$\Sigma:=I\times\Gamma$,

$\varSigma_{\rm el}$  
\REFF{first}  Piola-Kirchhoff stress,

$\varrho$ mass density (constant), 








$F=\nabla y$ deformation gradient, 

$F_{\rm el}$ elastic 
\REFF{part of $F$},

$E_{\rm el}$ elastic Green-Lagrange 
\REFF{strain},


$c_{\rm v}(\theta)$ heat capacity,

$\Frakk(
\zeta,\theta)$ heat-conductivity tensor,

$\mathscr{K}(\PP,\phi,\zeta,\theta)$ pull-back of $\Frakk(\zeta,\theta)$

$\vartheta$  rescaled temperature, 

$\eta$ entropy \REFF{(per unit reference volume)}, 


$m=m(\alpha,\phi)$ Biot modulus,

$\beta$ Biot coefficient,

$\lambda=\lambda(\alpha,\phi)$ first Lam\'e coefficient, 

$\GG=\GG(\alpha,\phi)$ shear modulus, 

$p_{\rm por}$ pore pressure,

$p_{\rm age}$ driving pressure for aging,

$p_{\rm eff}$ driving pressure for porosity,

\end{minipage}
\hfill 
\begin{minipage}[t]{0.53\linewidth}
\small

$\psi$ free energy \REFF{(in the reference configuration)},

$\FM$ mechanical part of $\psi$,

$\FT$ thermal part of $\psi$,

$\varSigma_{\rm in}$  driving stress for the plastification,

$\PR$ a placeholder for plastic rate $\DT\PP\PP^{-1}$,


$\mathfrak{R}$ dissipation potential for plastification,

$\mathfrak{D}$ dissipation potential for damage/porosity,


$\gamma=\gamma(\alpha,\phi)$ non-Hookean 
\WE{elastic modulus}, 

$\PORO=\PORO(\phi)$ porosity spherical strain influence,



$r$ dissipated mechanical energy rate,




$\Frakm=\Frakm(\alpha,\phi
)$ hydraulic conductivity, 

$\mathscr{M}(\PP,\alpha,\phi)$ pull-back of $\Frakm(\alpha,\phi)$,



$g$ gravity force in the actual space configuration,

$y_\flat$ external displacement loading,

$N$ constant of the elastic support, 


$\mu_\flat$ external chemical potential,

$M$  permeability at the boundary,

$K$ boundary heat-transfer coefficient,

$\theta_\flat$ external temperature,


$\chi$ specific stored energy of damage,

$\kappa_0$, $\kappa_1$, $\kappa_2$, $\kappa_3$, $\kappa_4$
length-scale coefficients,

$\tauR$ relaxation time for chemical potential.
\end{minipage}
}
\vspace{.4em}
\caption{Summary of the basic notation used through the paper.}
\label{Tab_Notation}
\end{table}
The model will result by combining momentum and energy
conservation with the dynamics of internal variables. In order to
specify the latter and provide constitutive relations, we introduce a
free energy and a dissipation (pseudo)potential in the following subsections.

\subsection{Small-strain mechanical stored energy}
 A crucial novelty of the present modelization is that of dealing with
finite strains. In order to motivate our assumptions on the mechanical stored 
energy in the coming Subsection \ref{sec:mse}, let us comment on a classical 
choice in the small-strain regime, namely 
\begin{align}
\frac12\lambda(\alpha)
I_1^2
+\GG(\alpha)I_2
-\gamma(\alpha)I_1\sqrt{I_2}
+\frac12m
|\beta I_1{-}\zeta{+}\phi|^2 
\label{ansatz} 
\end{align}
with $\lambda(\alpha)=\lambda_0, \
\GG(\alpha)=\GG_0-\alpha\GG_\mathrm{r},\
\gamma(\alpha)=\alpha\gamma_\mathrm{r}$, and with $I_1={\rm
  tr}\,e_{\rm el}$ and $I_2=|e_{\rm el}|^2$  and 
  $e_{\rm el}$  
\WE{denoting} the elastic part of the small 
strain. \REFF{When the undamaged-rock initial condition $\alpha_0=0$ is considered,
the values $\lambda_0$ and $\GG_0$ are the initial values 
of the elastic moduli $\lambda(\alpha)$ and $\GG(\alpha)$ 
in the rock (disregarding porosity, which is later considered), while 
$\GG_\mathrm{r}>0$ and $\gamma_\mathrm{r}>0$ are suitable constants (in MPa).} 
For $d=3$,  the so-called strain invariant ratio $I_1\sqrt{I_2}$
varies from $-\sqrt{3}$ for isotropic compaction to $\sqrt{3}$
 for  isotropic dilation. The Biot $m$-term  is an
extension of   the usual isotropic  response of a 
Lam\'e material with constants $\lambda$ and $\GG$ (\WE{latter} also called 
the shear modulus).  Such extension  was suggested in 
\cite{Biot41GTTS} and later augmented by the nonlinear (also called 
non-Hookean)  $\gamma$-term   in \cite{Biot73NSRP} at small
strains. The special form  of this last term  was suggested in 
\cite{LyaMya84BECS} (alternatively  considered as 
$-\gamma I_1\sqrt{I_2{-}I_1^2/3}$ in \cite{LRWS97NEBD}), 
validated, and  used in series of works 
\cite{HaLyAg04EDPP,HaLyAg05RDND,LyaBeZ08SREA,LyaBeZ09EGMP,LyaHam07DEFF,LyZhSh15VPDM}.
 The reader is referred to \cite{HaLyBZ11ESED} for a comprehensive
discussion on such choices.  
 
Note that  the above-introduced mechanical stored energy  is 2-homogeneous in terms of $e_{\rm el}$
and  that  the  function $\gamma$ in the 
$\gamma$-term  makes \REF{it} nonconvex if the damage  parameter 
$\alpha$ is sufficiently  large.  This  induces a loss of positive-definiteness
of the Hessian of \eqref{ansatz} which is
intended to model the loss  
of stability of the rocks under damage, cf.\ \cite{LyaBeZ08SREA}.  
 From the mathematical standpoint, this feature makes the analysis
challenging as   both coercivity and monotonicity of
the driving force  fails. This seemingly prevents  any rigorous
existence theory even  for short times,  due to possible stress concentration. 
A possible way out from this obstruction was proposed  in \cite{Roub17GMHF} by 
considering  nonsimple materials and by a regularization of the stored energy. 
An analogous regularization will be here considered. Indeed, we will replace 
the term $\gamma(\alpha)I_1\sqrt{I_2}$ by a bounded term 
$\gamma(\alpha)I_1\sqrt{I_2}/(1{+}\epsilon I_2)$ with a small, user-defined 
parameter $\epsilon>0$. 


In the small-strain setting, the following additive decomposition of
the total small strain is  often considered
\begin{align}\label{split-additive}
e(u)=e_{\rm el}+e_{\rm pl}
-\PORO(\phi)\bbI.
\end{align}
 Here, $e_{\rm pl}$ is the trace-free  plastic-strain tensor and 
$\PORO:[0,1]\to\R$ represents the pore volume (note that this is taken as
$\PORO(\phi)=-\phi/3$ in \cite{HaLyAg04EDPP,LyaHam07DEFF}). In 
\cite{LyaBeZ08SREA}, $e_{\rm pl}$ is eventually decomposed \REF{into} the sum of 
a damage-related inelastic strain and a creep-induced ductile strain 
(with a Maxwellian viscosity  of the order of  $10^{22\pm2}$Pa\,s), a
distinction which we neglect here.

\subsection{Mechanical stored energy}\label{sec:mse}
A focal point of our model is to move from the small to the finite strain 
situation. In particular, by replacing the small strain $e_{\rm el}$  with the 
elastic \REFF{Green-Lagrange} strain 
\REFF{$E_{\rm el}=\frac12(F_{\rm el}^\top F_{\rm el}^{}-\bbI)$}, we correspondingly 
consider the mechanical stored energy (compare with \eqref{ansatz}) 
\WE{as a function of the elastic strain $F_{\rm el}^{}$ as}
\begin{align}\label{ansatz+}
\FM&=\FM(F_{\rm el}\WE{,\PP},\alpha,\phi,\zeta)=\frac12\frac{\lambda(\alpha,\phi)
I_1^2}{\sqrt[4]{1{+}\epsilon I_2}}
+\frac{\GG(\alpha,\phi)I_2}{\sqrt[4]{1{+}\epsilon I_2}}
-\gamma(\alpha,\phi)\frac{I_1\sqrt{I_2}}{1{+}\epsilon I_2}
\\&\nonumber\qquad\qquad\qquad\qquad\qquad
\WE{+\varpi(\det\PP)}
+\frac12m(\alpha,\phi)
\frac{|\beta I_1{-}\zeta{+}\phi|^2}{\sqrt[4]{1{+}\epsilon I_2}}
+\chi(\alpha)
\end{align}
where now 
$I_1={\rm tr}\,E_{\rm el}$ and
$I_2=|E_{\rm el}|^2$
with $E_{\rm el}=\displaystyle{\frac{F_{\rm el}^\top F_{\rm el}-\bbI}2}$
so that $I_1\sqrt{I_2}={\rm tr}\,E_{\rm el}|E_{\rm el}|$. \WE{The $\varpi$-term 
is a modelling ansatz to ensure the plastic deformation to
be nearly isochoric, i.e.\ $\det\PP\sim1$. In combination with the
plastic-gradient term, this ensures local invertibility of the
plastic strain.
Formally, such a term acting on 
$\PP$ is in the position of isotropic hardening (typically occurring in
metals). Such hardening effect is indeed not relevant in modelling of
rocks or soils.  It should be
emphasized that, as here it controls only the volumetric part of $\PP$, 
it does not cause any undesired hardening effects when plastifying the rocks 
in a isochoric way.} The term $\chi(\alpha)$ in 
the right-hand side of \eqref{ansatz+} is the energy of damage and contributes 
an additional driving force for healing if $\chi'>0$. \REFF{The
  $\chi$-term can be microscopically interpreted as an extra energy contribution related with microvoids
or microcracks, arising due to macroscopical damage. This term is indeed a 
stored energy, although if damage would be unidirectional (without healing),
this energy would be effectively dissipated. Even in this case,
however, it would not contribute to the 
heat production, differently from truly dissipative terms. Also let us note that,} for $\epsilon=0$, \eqref{ansatz+} is an 
obvious analog of \eqref{ansatz}. Henceforth, we will however stick with a 
small but fixed $\epsilon>0$ in \eqref{ansatz+}. This yields a 3rd-order
polynomial growth of $\FM$ with respect to $F_{\rm el}$, which in turn
ensures that its derivative  has a  2nd-order polynomial 
growth.  In
particular,  all  driving forces  of the system, to be
defined in  (\ref{stresses}b,d) below, will turn
out to  belong to  $L^2$ spaces. Before moving on, let us
remark that the above choice of $\FM$ could be generalized, as long as
growth and smoothness properties are conserved. We shall however stick
to \eqref{ansatz+} for the sake of comparison with the small-strain
theory. 

A possible choice for the dependence of the nonlinearities
in \eqref{ansatz+} is 
\begin{subequations}\label{porosity-ansatz}
\begin{align}
&\lambda(\alpha,\phi)=(\lambda_0-\alpha\lambda_{\rm r})
\big(1{-}\phi/\phi_{\rm cr}\big),\ \ \ \ &&
\\&
\GG(\alpha,\phi)
=\big(\GG_0-\alpha\GG_\mathrm{r}\big)\big(1{-}\phi/\phi_{\rm cr}\big),\ \ \ \ 
\\&\gamma(\alpha,\phi)=\alpha\gamma_\mathrm{r}\big(1{-}\phi/\phi_{\rm cr}\big),
\ \ \ \ 
\\&m(\alpha,\phi)=\alpha m_0\big(1{-}\phi/\phi_{\rm cr}\big),
\end{align}
\end{subequations}
cf.\ \cite{HaLyAg04EDPP,LyaHam07DEFF},
where $\phi_{\rm cr}$ denotes the porosity upper bound in which the 
material loses its stiffness. \REFF{A typical value of $\phi_{\rm cr}$ 
used in geophysical applications is rather high, strong sandstones 
(e.g.\ Berea) may be $\sim$20\% while in other rocks it migh be even 
30\% or 40\%.
As for $\chi(\cdot)$, this direct damage energy is not considered
in the mentioned geophysical literature 
but it has a clear interpretation (as already explained) and may be a 
reasonable source of healing in addition to that healing due to 
$-\alpha \lambda_{\rm r}$ and $-\alpha\GG_\mathrm{r}$ terms in 
(\ref{porosity-ansatz}a,b). Besides, this term 
may also contribute to localization of damaged regions, which is 
routinely used in fracture mechanics under the name a ``phase-field fracture''.} In the case of an undamaged nonporous rock (i.e.\ 
$\epsilon=\alpha=\phi=0$  so that $\gamma(\alpha,\phi) =
m(\alpha,\phi) = \chi(\alpha)=0$ as well),  the  mechanical
stored energy \eqref{ansatz+} reduces to 
${\lambda_0}({\rm tr}E_{\rm el})^2/2+\GG_0|E_{\rm el}|^2$, namely to 
\WE{the} classical 
{\it St.\,Venant-Kirchhoff } material. As $\FM$ depends on the elastic 
Cauchy-Green tensor $F_{\rm el}^\top F_{\rm el}$ \WE{and
  $\PP^\top\PP$} rather than on $F_{\rm el}$ \WE{and $\PP$, so that}
the mechanical energy is \WE{both} {\it frame-} \WE{and {\it plastic-indifferent}}, namely 
   \begin{align}
    &\forall R_1,\WE{R_2}\in{\rm SO}(d):\ \ \FM(R_1F_{\rm
      el}\WE{,R_2\PP},\alpha,\phi,\zeta)= \FM(F_{\rm
    el}\WE{,\PP},\alpha,\phi,\zeta), \label{frame-indif}  
  \end{align} 
Here, we used the notation ${\rm SO}(d)$ for the matrix group 
${\rm SO}(d):=\{\WE{R}\in\R^{d\times d},\ \WE{RR}^\top=\WE{R}^\top 
\WE{R}=\mathbb{I},\ \det\WE{R}=1\}$ where the superscript $\top$ stands for 
transposition and $\mathbb{I}$ is the identity matrix. 

The additive decomposition \eqref{split-additive} from the
small-strain case is no longer available and one has to replace it
with the standard \WE{Kr\"oner-Lee-Liu} 
multiplicative decomposition \cite{Kron60AKVE,LeeLiu67FSEP} 
\begin{align}\label{split}
F=F_{\rm el}\PP S\ \ \text{ with }F=\nabla y\ \text{ and }\ S=S(\phi)=
\bbI/\PORO(\phi). 
\end{align}
Here, the nonlinearity $\PORO=\PORO(\phi)$  is related to the
stress-free \REFF{isotropic shrinkage}
of the specimen at given porosity $\phi$. 
\REFF{This corresponds to an expansion of volume $1/\PORO^d(\phi)$-times
in the stress-free state; note that 
$\det F=(\det F_{\rm el})(\det\PP)/\PORO^d(\phi)
\sim1/\PORO^d(\phi)$ because $\det\PP\sim1$ and 
also $\det F_{\rm el}\sim 1$ since $E_{\rm el}$ is assumed small
so that $\det F_{\rm el}=\det(F_{\rm el}^\top F_{\rm el})^{1/2}
=\det(2E_{\rm el}+\bbI)^{1/2}\sim(\det\bbI)^{1/2}=1$.
Using $\PORO(\phi)$ in the position of shrinkage rather than expansion in 
\eqref{split} corresponds to the negative sign in \eqref{split-additive}
and gives a simpler formula because $\PORO(\theta)$ occurs in $F_{\rm el}$
instead of $1/\PORO(\theta)$.}
%
%
 Our basic modeling assumption \WE{simplifying the model as far as 
the formulation and the analysis (cf.\ also Sect.\,\ref{sec-concl})}
is that the elastic part of the \REFF{Green-Lagrange} strain is small, 
namely $\REFF{E_{\rm el}=\frac12(F_{\rm el}^\top F_{\rm el}^{}-\mathbb I)\sim0}$, 
and, correspondingly, large deformations are accommodated by the inelastic 
term. \REFF{Yet, the large rotations are naturally allowed, 
so we do not assume directly $\nabla y\sim\PP S$ which might be too
restrictive in some geophysically relevant situations.}

 In addition to the already mentioned nonconvex $\gamma$-term, the 
 geometrically nonlinear setting of \eqref{split}   induces  additional 
nonconvexity of $\FM$.  
On the other hand, note that  $\FM $   is strongly convex
in terms of the water content $\zeta$.  This  makes  the model
amenable to a mathematical discussion even without considering
nonlocal contributions (gradient terms) for $\zeta$, which   would
lead to \WE{the} Cahn-Hilliard  dynamics. This feature will be used later
in order to deduce the   strong
convergence of the gradient of the chemical potential (i.e.\ of the
pore pressure).   

The multiplicative decomposition \eqref{split} allows to express
the free energy in terms of the total strain 
tensor and inelastic/ductile strains  via the substitution 
$F_{\rm el}=F\PP^{-1}S^{-1}(\phi) \WE{= \PORO(\phi)F\PP^{-1}}$. 
In addition, the mechanical stored energy 
will be augmented by gradient terms and a thermal contribution  $\FT$  
(considered for simplicity  to depend solely on temperature, i.e., thermal 
expansion which is not a 
\REF{dominant} effect in geophysical models is here neglected). By integrating 
\REF{over} the reference configuration  $\Omega$ with $F=\nabla y$, the 
{\it \REFF{total} free energy} of the body is expressed by   
\begin{align}
&\Psi(\nabla y,\PP,\alpha,\phi,\zeta,\theta)
=\PsiM(\nabla y,\PP,\alpha,\phi,\zeta)+\PsiT(\theta)
\label{free-energy+}
\\\nonumber
&\ \ \ \text{ with }\ \PsiT(\theta)=\int_\Omega\!\FT(\theta)\,\d x
\\\nonumber
&\ \ \ \text{ and }\ \ \PsiM(\nabla y,\PP,\alpha,\phi,\zeta)
=\int_\Omega\FM(\PORO(\phi)\nabla y\,\WE{\PP^{-1}}\!\WE{,\PP},\alpha,\phi,\zeta)
+\frac12\kappa_0\big|\nabla^2y\big|^2
\\[-.3em]&
\hspace{3em}
+\frac1q\kappa_1|\nabla \PP|^q
+\frac12\kappa_2|\nabla\alpha|^2+\frac12\kappa_3|\nabla\phi|^2
+\frac12\kappa_4|\nabla\zeta|^2
+\delta_{[0,1]}^{}(\zeta)\,\d x,
\nonumber
\end{align}
 where additional gradient terms are considered. In particular,
the $\kappa_0$-term qualifies the material as {\it 2nd-grade
  nonsimple}, also called {\it multipolar} or {\it complex}, see the
seminal  \cite{Toup62EMCS}  and
\cite{FriGur06TBBC,MinEsh68FSTL,Podi02CISM,PoGiVi10HHCS,Silh88PTNB,TriAif86GALD}.
The exponent $q$ in the $\kappa_1$-term  is given and fixed to be
larger than $d$,  which eases  some points of the analysis. Note however that the
choice $\kappa_1(\nabla \PP) |\nabla \PP|^2$ for some $ \kappa_1 (\nabla
\PP)\sim 1+ |\nabla \PP|^{q-2}$ could be considered as well. The
gradient terms in $\alpha$, $\phi$, and $\zeta$ are intended to
describe nonlocal effects  and effectively encode the emergence of
length scales associated with damage, porosity, and water-content
profiles, respectively. 

 The symbol 
$\delta_{[0,1]}(\cdot)$ denotes   the indicator function $\delta_{[0,1]}(\zeta)=0$ 
if $0\le\zeta\le1$  and   $\delta_{[0,1]}(\zeta)=\infty$ 
elsewhere.  
This indicator function encodes  the constraint $0\le\zeta\le1$.
The {\it frame-} \WE{and {\it plastic-}}indifference of the  mechanical stored 
energy \eqref{frame-indif} translates in terms of $\Psi$ as
\begin{align*}
\forall R_1,\WE{R_2}\in{\rm SO}(d):\ \ \Psi({R_1}\nabla y, \WE{R_2} \PP,\alpha,\phi,\zeta,\theta)=\Psi(\nabla y,\PP,\alpha,\phi,\zeta,\theta).
\end{align*}
In particular let us note that the gradient terms 
are  frame-indifferent  as well.

%
 
The partial functional derivatives of $\Psi$
 give origin to  corresponding 
driving forces.  We use the symbol $\partial_w$ to indicate both
differentiation with respect to the variable $w$ of a smooth function
or functional 
or subdifferentiation of a convex function or functional. The 
second  Piola-Kirchhoff 
stress $\varSigma_{\rm el}$,  here  augmented by a contribution 
arising from the  gradient $\kappa_0$-term, is defined   as 
\begin{subequations}\label{stresses}\begin{align}
&\varSigma_{\rm el}=
  \partial_{\nabla y}   \Psi  =
  \PORO(\phi) \partial_{F_{\rm el}}^{} \FM  (\PORO(\phi)\nabla
 y\WE{\PP^{-1}}\!,\alpha,\phi,\zeta) \WE{\PP^{-T}}
\label{elastic-stresses}
-\kappa_0{\rm div}\nabla^2y.
\ \ \ \ 
\intertext{
Furthermore, 
 the
driving stress for the plastification, again involving   a contribution 
arising from the gradient $\kappa_1$-term, reads}
&\label{def-of-Sin}
\varSigma_{\rm in}=\partial_{\PP}^{}\Psi=
\PORO(\phi) \nabla y^\top  \partial_{F_{\rm el}}\FM
(\PORO(\phi)\nabla y\WE{\PP^{-1}}\!,\alpha,\phi,\zeta): \partial_\PP
\WE{\PP^{-1}}  
\\\nonumber
&\hspace{14em}
\WE{+\varpi'(\det\PP)\WE{\PP^{-T}}}-{\rm div}(\kappa_1|\nabla\PP|^{q-2}\nabla\PP)
.
\intertext{Here and in the following  we use the (standard) notation ``$\,\cdot\,$'' and ``$\,:\,$'' 
and ``$\,\Vdots\,$''  for the  contraction 
product of vectors, 2nd-order, and 3rd tensors, respectively. As  $\partial_\PP
\WE{\PP^{-1}} $ is a 4th-order tensor, the product 
$\PORO(\phi) \nabla y^\top\partial_{F_{\rm el}}\FM
(\PORO(\phi)\nabla y\WE{\PP^{-1}}\!,\alpha,\phi,\zeta): \partial_\PP
\WE{\PP^{-1}}$   turns out
to be a 2nd-order tensor, as expected.  The 
thermodynamical driving {\it pressure} for damage  is }
\label{p-dam}
&p_{\rm age}=
 \partial_\alpha \Psi=
 \partial_\alpha\FM   (\PORO(\phi)\nabla y\WE{\PP^{-1}}\!,\alpha,\phi,\zeta)-\kappa_2\Delta\alpha,
\intertext{and the driving  force  for porosity-evolution 
(a so-called effective pressure)  is }
&p_{\rm eff}= \partial_\phi \Psi =
 {\rm tr} \big( \PORO'(\phi)\WE{\PP^{-T}}\nabla y^\top \partial_{F_{\rm el}}^{}\FM
(\PORO(\phi)\nabla y\WE{\PP^{-1}}\!,\alpha,\phi,\zeta) \big) 
\\&\hspace{5.8em}
+  \partial_\phi \FM   (\PORO(\phi)\nabla y\WE{\PP^{-1}}\!,\alpha,\phi,\zeta)
-\kappa_3\Delta\phi.\nonumber
\intertext{Analogously, we also identify the pore pressure 
$p_{\rm por}$ as}\label{p-por}
&
p_{\rm por}\in\partial_\zeta\Psi=
 \partial_\zeta \FM  (\PORO(\phi)\nabla y\WE{\PP^{-1}}\!,\alpha,\phi,\zeta)
-\kappa_4\Delta\zeta+\mathfrak{N}_{[0,1]}^{}(\zeta)
\\&\hspace{4.8em}\nonumber
=m(\alpha,\phi)\frac{\zeta{-}\phi{-}\beta I_1}{\sqrt[4]{1{+}\epsilon I_2}}
-\kappa_4\Delta\zeta+\mathfrak{N}_{[0,1]}^{}(\zeta),
\end{align}\end{subequations}
where $\mathfrak{N}_{[0,1]}^{}(\zeta)$ is the normal cone  to the
interval $[0,1]$  at $\zeta$.   All variations of $\Psi$ above are taken with
respect to the corresponding $L^2$ topologies.

\subsection{Thermodynamical system}\label{sec:thermo}
The 
entropy $\eta$, the heat capacity $c_{\rm v}$, and the thermal part
$\vartheta$  of the internal energy \REFF{(per unit reference volume)} 
are classically recovered as 
\begin{align}&\eta=-\psi_\theta'=-\FT'(\theta),\ \ \ \ 
c_{\rm v}=-\theta\psi_{\theta\theta}''=-\theta\FT''(\theta),\ \ 
\text{ and }\ \ \vartheta=\FT(\theta)-\theta\FT'(\theta).
\label{p-eta-cv}
\end{align}
Note  in particular that 
$\DT\vartheta=\FT'(\theta)\DT\theta-\DT\theta\FT'(\theta)-\theta\FT''(\theta)\DT\theta=c_{\rm v}(\theta)\DT\theta$. The {\it entropy equation} reads as
\begin{align}\label{ent-eq}
\theta\DT\eta+{\rm div}\,j=\text{dissipation rate}.
\end{align}
We assume the heat flux $j$  to be  governed by the {\it Fourier} law
$j=-\mathscr{K}\nabla\theta$ where is $\mathscr{K}$ the heat-conductivity 
tensor. Substituting $\eta$ from \eqref{p-eta-cv} into \eqref{ent-eq}, we 
arrive \REF{at} the heat-transfer equation
\begin{align*}
c_{\rm v}(\theta)\DT\theta-{\rm div}(\mathscr{K}\nabla\theta)=\text{dissipation rate}.
\end{align*}
Note that,  $c_{\rm v}$ depends of temperature only as so does
$\FT$. 

The 
\WE{water-content} gradient (i.e.\ the $\kappa_4$-term) describes capillarity 
effects and it is standardly referred to as the Cahn-Hilliard model 
\cite{CahHil58FEUS}, $\por$ is the chemical potential and \eqref{p-por} 
corresponds to diffusion governed by the (generalized) Fick-Darcy law. Note 
that this simplified model for a  stiff poroelastic matrix interacting with
a moving fluid is largely accepted in  the geophysical context 
\cite{Raja07HAMF}. In order to cope with the direct coupling of $\zeta$ with 
$\theta$ in \eqref{ansatz}, we consider some viscous dynamics,
following the original Gurtin's ideas \cite{Gurt96GGLC}, cf.\ also 
\cite{BoDrSc03GSGC,EllGar96CHED,GriNov99VCHE,HeKrRoRo??TDPF,Ross05TCGV}.
This involves some relaxation time $\tauR>0$ and a contribution 
$\tauR\DT\zeta^2$ to the dissipation rate.

\REF{In summary}, the  model consists of a system of semilinear equations of 
the form
\begin{subequations}\label{system}
\begin{align}\label{momentum-eq}
&\varrho\DDT y={\rm div}
\,
\varSigma_{\rm el}
+
g(y),\!\!&&\text{({\sf momentum equilibrium})}
\\&\label{flow-rule-pi}
\partial_\PR^{}\mathfrak{R}\big(\alpha,\phi,\theta;\DT\PP\PP^{-1}\big)
+
\,\varSigma_{\rm in}\PP^\top
=0,
&&\text{({\sf flow rule for inelastic strain})}
\\[-.1em]&
\partial_{(\DT\alpha,\DT\phi)}\mathfrak{D}\Big(\!\alpha,\phi,\theta;\Big(\!\begin{array}{c}\DT\alpha\\[-.3em]\DT\phi\end{array}\!\Big)\Big)
+\Big(\!\!\begin{array}{c}
p_{\rm age}\!
\\[-.1em]
p_{\rm eff}\,
\end{array}\!\!\Big)
=\Big(\!\!\begin{array}{c}0\!\!
\\[-.1em]
0\!\!\end{array}\Big)
,\!\!&&\text{({\sf flow rule for damage/porosity})}
\label{flow-rule}
\\[-.3em]\label{system-Darcy}
&\DT\zeta=\EE{\rm div}\big(\mathscr{M}(\PP,\alpha,\phi)
\nabla \por^{}\big)
,
&&\text{({\sf water-transport equation})}
\\[-.3em]\label{system-Cahn-Hilliard}
&\por=
p_{\rm por}+\tauR\DT\zeta,
&&\text{({\sf equation for chemical potential})}
\\
&c_{\rm v}(\theta)\DT\theta
={\rm div}\big(\mathscr{K}(\PP,\phi,\zeta,\theta)\nabla\theta\big)+r
\!\!
&&\text{({\sf heat-transfer equation})}
\label{system-heat}
\\\label{system-diss}
&\qquad\text{with }\ r=r\big(\PP,\alpha,\phi,\theta;\DT\PP,\DT\alpha,\DT\phi,\DT\zeta,\nabla \por^{}\big)
\\[-.4em]\nonumber
&\hspace{5em}=
\partial_\PR^{}\mathfrak{R}\big(\alpha,\phi,\theta;\DT\PP\PP^{-1}){:}(\DT\PP\PP^{-1})
+
\partial_{(\DT\alpha,\DT\phi)}\mathfrak{D}\Big(\alpha,\phi,\theta;\Big(\!\begin{array}{c}\DT\alpha\\[-.1em]\DT\phi\end{array}\!\Big)\Big)
{\cdot}\Big(\!\begin{array}{c}\DT\alpha\\[-.1em]\DT\phi\end{array}\!\Big)
\hspace{-15em}\\[-.4em]&\hspace{5em}
+\tauR\DT\zeta^2
+\mathscr{M}(\PP,\alpha,\phi)\nabla \por^{}{\cdot}\nabla \por^{},
\hspace{-0em}
&&
\text{({\sf heat-production rate})} 
\nonumber\end{align}
\end{subequations}
where $\mathfrak{R}=\mathfrak{R}\big(\alpha,\phi,\theta;\PR)$ is the 
pseudopotential related to dissipative forces of visco-plastic origin ($\PR$ is 
the placeholder for the 
\REFF{rate of plastic strain}
$\DT\PP\PP^{-1}$),   
and 
$\mathfrak{D}(\alpha,\phi,\theta;\cdot)$ 
is the dissipation potential  related to   damage and porosity evolution.

The {\it effective transport matrices}
$\mathscr{K}$ and $\mathscr{M}$ are to be related with the 
hydraulic-conductivity and the heat-conductivity symmetric tensors 
$\Frakm=\Frakm(\alpha,\phi)$ and $\Frakk=\Frakk(\zeta,\theta)$
which are given material properties. 
 The need for such effective quantities stems form the fact that 
driving forces are to be considered Eulerian in nature, so that a
pull-back to the reference configuration is imperative. A first choice
would then be  
\begin{subequations}\label{M-K-pull-back}
\begin{align}\label{M-pull-back}
\mathscr{M}(F,\alpha,\phi)&=(\det F)F^{-\top}\Frakm(\alpha,\phi) F^{-1}
=(\Cof F)\Frakm(\alpha,\phi)F^{-1}
=\frac{1}{{\det F}}(\Cof F)\Frakm(\alpha,\phi)\Cof F^\top
\\\label{K-pull-back}
\mathscr{K}(F,\zeta,\theta)
&=\frac{1}{{\det F}}{(\Cof F)\Frakk(\zeta,\theta)\Cof F^\top} .
\end{align}\end{subequations}
These are just  usual pull-back transformations of 2nd-order covariant 
tensors\REFF{, cf.\ also Remark~\ref{rem-Lag-vs-Euler} below for some more 
discussion}.  
%
\WE{Let us} recall that our modeling assumption is that  
\WE{$E_{\rm el}$}
is small so that 
\WE{ $F^\top F\sim S^\top(\phi)\PP^\top\PP S(\phi)=\PP^\top\PP/\PORO^2(\phi)$.}
Thus, by replacing $F$ by $\PP S(\phi)=\PP/\PORO(\phi)$ 
\WE{(see Remark~\ref{rem-isotropic}
below for some more discussion) 
and using the specific homogeneity of the determinant and the cofactor}, 
relations \eqref{M-K-pull-back} can be rewritten as 
\begin{subequations}\label{M-K-pull-back+}
\begin{align}\label{M-pull-back+}
&\mathscr{M}(\PP,\alpha,\phi)\ =\ 
\PORO(\phi)^{2-d}\frac{\WE{\PP^{-T}}\Frakm(\alpha,\phi)\WE{\PP^{-1}}}{\WE{\det\PP}},
\\&\mathscr{K}(\PP,\zeta,\theta)\ =\ 
\PORO(\phi)^{2-d}\frac{\WE{\PP^{-T}}\Frakk(\zeta,\theta)\WE{\PP^{-1}}}{\WE{\det\PP}}.
\end{align}\end{subequations}
These expressions bear the advantage of being independent of $(\nabla y)^{-1}$, 
which turns out useful in relation with estimation and passage to the limit
arguments, cf.\ \cite{KruRou18MMCM,RouTom??TMPE}.

Note that  the right-hand side of  \eqref{momentum-eq} features
the pull-back $g{\circ}y$ of the {\it actual} gravity force
$g:\R^d\to\R^d$. This allows us to  consider a spatially inhomogoneous
gravity, a generality which could turn out to be sensible at geophysical
scales. 





The plastic flow rule \eqref{flow-rule-pi}  complies with the
so-called {\it plastic-indifference} requirement. Indeed, the evolution is 
\REFF{insensitive} to prior plastic deformations, for the stored energy 
and the dissipation potential  
$$
\widehatFM(F,\PP,\alpha,\phi,\zeta)
:=\FM(F_{\rm el},\alpha,\phi,\zeta)
\ \ \ \ \text{ and }\ \ \ \ \widehat{\mathfrak{R}}(\PP,\alpha,\phi,\theta;
\DT\PP) = \mathfrak{R} (\alpha,\phi,\theta;\DT \PP \PP^{-1})
$$
\WE{respect}
the 
invariances 
$\widehatFM(F\widetilde\PP,\PP\widetilde\PP,\alpha,\phi,\zeta)=
\widehatFM(F,\PP,\alpha,\phi,\zeta)$
and $\widehat{\mathfrak{R}}(\PP\widetilde\PP,\alpha,\phi,\theta;\DT\PP\widetilde\PP) 
=\widehat{\mathfrak{R}}(\PP,\alpha,\phi,\theta;\DT\PP)$
for any $\widetilde\PP\in{\rm SO}(d)$ \REFF{meaning the mentioned  prior plastic deformation},
cf.\ e.g.\ \cite{Miel02FELG,MiRoSa??GERV,RouSte??TELS}.
 In particular,  we can equivalently test the flow rule
\eqref{flow-rule-pi} by $\DT\PP\PP^{-1}$ or rewrite it  as  
\begin{align}\label{flow-rule-pi+}
\partial_\PR^{}\mathfrak{R}\big(\alpha,\phi,\theta;\DT\PP\PP^{-1}\big)\PP^{-\top}
+\,\varSigma_{\rm in}
= \partial_{\DT\PP}\widehat{\mathfrak{R}}\big(\PP,\alpha,\phi,\theta;\DT\PP
\big) +\,\varSigma_{\rm in}  =0
\end{align}
and test it on by $\DT\PP$ obtaining 
\begin{align}\label{test-of-flow-rule}
\partial_\PR^{}\mathfrak{R}\big(\alpha,\phi,\theta;\DT\PP\PP^{-1}\big)
{:}\DT\PP\PP^{-1}=-\varSigma_{\rm in}\PP^\top\!{:}\DT\PP\PP^{-1}
=-\varSigma_{\rm in}\PP^\top\PP^{-\top}\!{:}\DT\PP=-\varSigma_{\rm in}{:}\DT\PP,
\end{align}
where we used also the algebra $AB{:}C=A{:}CB^\top$.

%
 %

 The system \eqref{system} has to be complemented by 
 suitable boundary and initial conditions.  As for the former
we prescribe 
\begin{subequations}\label{BC}
\begin{align}\label{BC-1}
&\varSigma_{\rm el}\nu-
\divS\big(\kappa_0\nabla^2y
\big)
+Ny=Ny_\flat(t),\ \ \ \ \ \ \ \kappa_0\nabla^2y{:}(\nu\otimes\nu)=0,
\\&
\mathscr M(\PP,\alpha,\phi)\nabla \por^{}{\cdot}\nu+M\por=M\mu_\flat(t),\ \ \ \ \ \ 
\mathscr K(\PP,\zeta,\theta)\nabla\theta{\cdot}\nu+K\theta=K\theta_\flat(t),\ \ \ \ 
\label{BC-2}
\\& \PP =\mathbb{I} \ \ \WE{\text{on} \  \Gamma_{\rm Dir} \subset \partial \Omega}, \ \ \ \ \ \ 
\kappa_2\nabla\alpha{\cdot}\nu=0,\ \ \ \ \ \ 
\kappa_3\nabla\phi{\cdot}\nu=0,\ \ \ \ \ \ 
\kappa_4\nabla\zeta{\cdot}\nu=0.
\label{BC-3}\end{align}\end{subequations}
Relations \eqref{BC-1} correspond to a Robin-type
  mechanical condition. In particular,  $\nu$ is the external normal at $\partial \Omega$, ${\rm div}_{_{\rm S}}$ denotes the surface divergence defined as a
trace of the surface gradient (which is a projection of the gradient on the 
tangent space through the projector $\bbI-\nu\otimes\nu$),
and $N$ is the elastic modulus of idealized {\it boundary} springs
(as often used in numerical simulations in geophysical models, cf.
e.g.\ \cite{LyaBeZ08SREA,LyaBeZ09EGMP}).
Similarly we prescribe  in \eqref{BC-2}  Robin-type boundary condition for the water flow
 where $M$ is  a boundary permeability and $\mu_\flat$  is the
 water chemical potential
in the external environment, and for temperature,  where $K$ is
 the boundary 
heat-transfer coefficient and $\theta_\flat$  is   the external temperature. 
Moreover, the $\kappa$-gradient  terms
require corresponding boundary conditions.  We assume $\PP$ to be
    the identity \WE{on an open subset $\Gamma_{\rm Dir} $ of
      $\partial \Omega$ having a
      positive surface measure}. This \WE{boundary condition} is chosen here for the sake of
    simplicity and could  be weakened by imposing \WE{the condition to
      be non-homogeneous $\PP= \PP_{\rm
        Dir}(t)$ and possibly time-dependent on $\Gamma_{\rm Dir}$} or even by a Neumann condition, this
    last requiring however a  more  delicate estimation argument\REFF{, cf.\ also Sect.~\ref{sec-concl}}.
All other boundary conditions are assumed to be of homogeneous Neumann-type in 
\eqref{BC-3}. Eventually, initial conditions read 
\begin{align}\label{IC}
&y(0)=y_0,\ \ \ \DT y(0)=v_0,\ \ \ \PP(0)=\PP_0,\ \ \ \alpha(0)=\alpha_0,\ \ \ 
\phi(0)=\phi_0,\ \ \ \theta(0)=\theta_0.
\end{align}

We shall comment \REF{on} the thermodynamic consistency of the full
model  \eqref{system}--\eqref{BC}--\eqref{IC}. This can be checked  
by testing the particular equations/inclusions in (\ref{system}a-e)
successively by $\DT y$, $\DT\PP\PP^{-1}$, $\DT\alpha$, $\DT\phi$, 
$\por$, and $\DT\zeta$.  By adding up these contributions and using
\eqref{test-of-flow-rule} we  obtain 
the mechanical energy balance 
\begin{align}\label{energy-conserv}&\!\!
\frac{\d}{\d t}\bigg(\int_\Omega\frac\varrho2|\DT y|^2
\,\d x
+\PsiM(\nabla y,\PP,\alpha,\phi,\zeta)
+\int_\Gamma \frac12N|y|^2\,\d S\bigg)
\\[-.3em]&\qquad
+\int_\Omega r\big(\PP,\alpha,\phi,\theta;\DT\PP,\DT\alpha,\DT\phi,\DT\zeta,\nabla \por^{}\big)\,\d x
+\int_\Gamma M\por^2\,\d S 
=\int_\Omega 
g(y){\cdot}\DT y\,\d x
+\int_\Gamma 
Ny_\flat{\cdot}\DT y+
M\mu_\flat\por\,\d S.
\nonumber\end{align}
Let us point out that, as usual,  this energy balance  can be
rigorously justified in case of smooth solutions only. Existence of
smooth solutions is however not guaranteed for $y$ lacks time regularity due
to the possible occurrence of shock-waves in the nonlinear hyperbolic system
(\ref{system}a).  
Also the power of the external mechanical load in \eqref{energy-conserv}, 
i.e.\ $y_\flat
{\cdot}\DT y$, is not well defined if $\nabla\DT y$ 
is not controlled.  We will hence  
 treat this term \REF{as a weak derivative in time}, \WE{using} 
the by-part integration in time, cf.\ \eqref{by-part-boundary}.

By adding to \eqref{energy-conserv} the space integral of the heat equation 
\eqref{system-heat} we obtain the total energy balance
\begin{align}\label{energy-conserv+}
&\hspace{-1em}\frac{\d}{\d t}\bigg(
\!\!\!\linesunder{\int_\Omega\frac\varrho2|\DT y|^2
+\vartheta\,\d x}{kinetic and heat}{energies in the bulk}\!\!\!
+\!\!\!\!\linesunder{\PsiM(\nabla y,\PP,\alpha,\phi,\zeta)_{_{_{_{_{_{_{}}}}}}}\!}{mechanical energy}{in the bulk}\!\!\!\!
+\!\!\!\!\!\!\!\!\!\linesunder{\int_\Gamma \frac12N|y|^2\d S_{_{_{_{_{_{_{}}}}}}}\!}{mechanical energy}{on the boundary}\!\!\!\!\!\!\bigg)
\\[-.2em]
&\hspace{-1em}=\!\!\!\!\!\linesunder{\int_\Omega 
g(y)\cdot\DT y\,\d x_{_{_{_{_{_{_{}}}}}}}\!}{power of}{gravity}\!\!\!\!\!
+\!\!\!\!\linesunder{
\int_\Gamma Ny_\flat\cdot\DT y\,\d S}{power of surface}{load on $\Gamma$}\!\!\!\!
+\!\!\!\!\linesunder{\int_\Gamma M(\por^{}{-}\mu_\flat)\,\d S}{flux of energy due}{to water flow thru $\Gamma$}\!\!\!\!
+\!\!\!\!\linesunder{\int_\Gamma K(\theta{-}\theta_\flat)\,\d S}{heat flux}{thru $\Gamma$}\!\!\!.\!\!\!\!
\!
\nonumber\end{align}

From \eqref{ent-eq} with the heat flux $j=-\mathscr{K}\nabla\theta$ and with 
the dissipation rate (=heat production rate) $r$
from \eqref{system-diss}, one can read the {\it entropy imbalance}
\begin{align}\label{ent-imbalance}
\frac{\d}{\d t}\int_\Omega\!\eta\,\d x
=\int_\Omega\!\frac{r+{\rm div}(\mathscr{K}\nabla\theta)}{\theta}\,\d x
=\int_\Omega\,\frac{r}{\theta}-\mathscr{K}\nabla\theta{\cdot}\nabla\frac1\theta\,\d x
+\int_\Gamma\frac{\mathscr{K}\nabla\theta}\theta{\cdot}\nu\,\d S
\\\qquad\quad=\int_\Omega\!\!\!\!\linesunder{\;\frac{r}{\theta}+\frac{\mathscr{K}\nabla\theta{\cdot}\nabla\theta}{\theta^2}}{entropy production}{rate in the bulk $\Omega$}\!\!\!\!\!\d x\ 
+\int_\Gamma\!\!\!\!\!\!\!\!\!\!\!\!\linesunder{\;\frac{K(\theta_\flat{-}\theta)}\theta}{entropy flux through}{the boundary $\Gamma$}\!\!\!\!\!\!\!\!\!\!\d S
\ge\int_\Gamma K\Big(\frac{\theta_\flat}\theta-1\Big)\theta\,\d S,
\nonumber
\end{align}
provided $\theta>0$ and $\mathscr{K}$ is positive semidefinite. In particular,
if the system is thermally isolated, i.e.\ $K=0$,
\eqref{ent-imbalance}  states that 
the overall entropy is nondecreasing in time.
This shows consistency with the 2nd law of thermodynamics.

Eventually, the 3rd thermodynamical law (i.e.\ non-negativity of
temperature),  holds  as  soon as   the initial/boundary conditions
are  suitably  qualified so that $r \geq 0$.  In fact,   we do not consider any 
adiabatic-type effects, which might cause cooling.  

 We conclude the presentation of the model with a number of
remarks and comments on  modeling choices and  possible extensions.  

\begin{remark}[{\sl Dissipation potential $\mathfrak{D}$}]\label{rem-D-potent}
\upshape
A specific form of the flow rule for the damage/poro\-sity 
\eqref{flow-rule}  can be chosen as  
\begin{subequations}\label{flow-damage+poro}\begin{align}\label{flow-damage+}
&\DT\alpha=\begin{cases}
c_0\big(\gamma_\mathrm{r}I_2(\xi{-}\xi_0)+\kappa_2\Delta\alpha\big)
&\text{if }
\gamma_\mathrm{r}I_2(\xi{-}\xi_0)+\kappa_2\Delta\alpha\ge0,
\\
c_1\mathrm{e}^{\alpha/c_2}\mathrm{e}^{b(\phi_0-\phi)}
\big(\gamma_\mathrm{r}I_2(\xi{-}\xi_0)+\kappa_2\Delta\alpha\big)
&\text{otherwise,}\end{cases}
\\&\label{flow-poro}\DT\phi=d(\phi)\big|p_{\rm eff}{+}\kappa_3\Delta\phi\big|^n
\big(p_{\rm eff}{+}\kappa_3\Delta\phi\big)
\end{align}\end{subequations}
 where  the so-called strain invariants ratio
$\xi:=I_1/\sqrt{I_2}$  is used,  $c_0$, $c_1$, and $c_2$ 
denote positive parameters,  and $\GG_\mathrm{r}$ and 
$\gamma_\mathrm{r}$  are  from \eqref{ansatz}. 
In particular, 
  $\xi_0$ is a   critical strain invariant ratio  thresholding
  damaging from  healing. This  flow rule leads to a
  dissipation 
potential  which is $(2,\frac{n+2}{n+1})$-homogeneous
in terms of the rates $(\DT\alpha,\DT\phi)$,  namely 
\begin{align}
\mathfrak{D}(\alpha,\phi;\DT\alpha,\DT\phi)
&=\displaystyle\frac{n{+}1}{n{+}2}d(\phi)^{-n-1}|\DT\phi|^{(n+2)/(n+1)}
+
\begin{cases}\displaystyle\frac1{2c_0}\DT\alpha^2
&\text{if }\DT\alpha\ge0,\\
\displaystyle\frac1{2c_1}  \mathrm{e}^{-\alpha/c_2
  -b(\phi_0-\phi)}  \DT\alpha^2&\text{if }\DT\alpha\le0.
\end{cases}\nonumber
\end{align}
 In the case $\DT \phi =0$,   the flow rule \eqref{flow-damage+}
has been used in \cite{LyaBeZ09EGMP}  (with $\kappa_2=0$) and
\cite[Formula~(25)]{LyaBeZ14CDBF}. Note that  
$\DT \alpha \mapsto \mathfrak{D}(\alpha,\phi_0;\DT \alpha, 0)$   is convex, degree-2 homogeneous, and 
differentiable at $\DT\alpha=0$.  This  suggests to call $\alpha$ 
{\it aging} (as indeed mostly used in the geophysical  literature) 
rather than damage.  In addition, parameter dependencies on
temperature, i.e. $\mathfrak{D}=\mathfrak{D}
(\alpha,\phi,\theta;\DT\alpha,\DT\phi)$ can also be considered, see
below. 
 By including the evolution of porosity as well,   a
non-dissipative  antisymmetric   coupling 
between  the two  flow rules in 
\eqref{flow-damage+poro} has been considered in
\cite{HaLyAg04EDPP,HaLyAg05DRDC,LyaHam07DEFF}. Such dissipation does
not   admit a  potential and does not control 
$\DT\phi$.  In particular, standard existence theories are not
applicable. A symmetric version of this coupling  has also been
proposed  for a similar model
with a granular-phase field instead of the porosity
\cite{LyaBeZ14CDBF}.  This would indeed admit
a potential and be amenable to variational solvability. 
\end{remark}

\REFF{
\begin{remark}[{\sl The transport tensors $\mathbb M$ and $\mathbb K$}]
\label{rem-Lag-vs-Euler}
\upshape
The Darcy and Fourier laws in \eqref{M-K-pull-back} are 
in the actual deformed configuration, and one expects 
to consider the transport coefficients $\mathbb M_\text{\sc a}$ and 
$\mathbb K_\text{\sc a}$ as a function of $y\in y(\Omega)$, while
the ``effective'' transport tensors $\mathscr{M}$ and $\mathscr{K}$ 
are in the reference Lagrangian coordinates.
In real situations, one must feed the model with transport coefficients
that are known for particular materials at the point $x\in\Omega$.
Then $\mathbb M_\text{\sc a}=\mathbb M_\text{\sc a}(y)$, which should be 
thought actually in the right-hand side of \eqref{M-pull-back}, can be chosen 
as $\mathbb M_\text{\sc a}(y)=\mathbb M(y^{-1}(y(x)))=\mathbb M(x)$.
If $\mathbb M_\text{\sc a}(y)$ depends also on the scalar internal variables 
(i.e.\ aging $\alpha^y$ and porosity $\phi^y$ considered also in 
$y(\Omega)$ rather than $\Omega$),
then this transformation applies similarly, i.e.\ 
using $\alpha^y(y(x))=\alpha(y^{-1}(y(x))=\alpha(x)$ and 
$\phi^y(y(x))=\phi(y^{-1}(y(x))=\phi(x)$, we obtain 
$\mathbb K=\mathbb K(x,\alpha,\phi)$ fully expressed in 
the Lagrangian reference configuration. The same applies to $\mathbb K$ 
in \eqref{K-pull-back}. 
\end{remark}
}

\begin{remark}[{\sl The \WE{isotropic} choice of $\mathbb M$}]
\label{rem-isotropic}
\upshape
The mobility $\mathbb M$ in the Darcy law in 
\REFF{configuration (considered eventually in the reference configuration 
as explained in Remark~\ref{rem-Lag-vs-Euler})}
is often considered  to be  isotropic, namely, $\mathbb
M=\kappa\bbI$ where $\kappa>0$ is the  so-called hydraulic
conductivity or permeability.  This amounts to about
$10^{-12}$m$^2$/(Pa\,s)  \cite{HaLyAg04EDPP,HaLyAg05DRDC} but may also
depend on porosity and/or damage as  
  $\kappa=\kappa(\phi)$ or $\kappa=\kappa(\alpha,\phi)$ 
with various phenomenologies  \cite{LyaHam07DEFF,LyZhSh15VPDM}.
\REFF{
In this isotropic case $\Frakk(\theta)=k(\theta)\mathbb{I}$, 
relation  
\eqref{M-pull-back} can also be written 
by using the right Cauchy-Green tensor $C$ as 
\begin{align*}
\mathscr{K}(F,\theta)=\hat{\mathscr{K}}(C,\theta) =\det C^{1/2} k(\theta)C^{-1}
\quad\text{with }\ \ C=F^\top F, 
\end{align*}
cf.\ \cite[Formula (67)]{DuSoFi10TSMF} or
\cite[Formula (3.19)]{GovSim93CSD2}.
In fact, the effective transport-coefficient tensor is a function of $C$ 
in general anisotropic cases as well, cf.\ \cite[Sect.\,9.1]{KruRou18MMCM}. In 
view of this, we now use our smallness assumption $E_{\rm el}\sim 0$, which 
yields only $F^\top F\sim S^\top(\phi)\PP^\top\PP S(\phi)
=\PP^\top\PP/\PORO^2(\theta)$, 
in order to infer that we can, in fact, substitute $F$ 
with  $\PP/\PORO(\theta)$
into \eqref{M-pull-back} as a good modelling ansatz, even though 
$F-\PP/\PORO(\theta)$ need not be small. Similar consideration 
holds for the heat transfer, too.}
\end{remark}

\section{Existence of weak solutions}
\label{sec-anal}

This section introduces the definition of weak solution to the
problem and brings to the statement of the our main existence result,
namely Theorem \ref{thm}. Let us start by fixing some notation. 

We will use the standard notation $C(\cdot)$ for the space of continuous bounded
functions, $L^p$ for Lebesgue  spaces, and $W^{k,p}$ for Sobolev spaces whose 
$k$-th distributional derivatives are in $L^p$. Moreover, we will use the 
abbreviation $H^k=W^{k,2}$ and,  for all $p\geq 1$, we let the conjugate
exponent $p'=p/(p{-}1)$  (with $p'=\infty$ if $p=1$), and 
$p^*$ for the Sobolev exponent $p^*=pd/(d{-}p)$ for $p<d$,
$p^*<\infty$ for $p=d$, and $p^*=\infty$ for $p>d$.
Thus, $W^{1,p}(\Omega)\subset L^{p^*}\!(\Omega)$ or 
$L^{{p^*}'}\!(\Omega)\subset W^{1,p}(\Omega)^*$=\,the dual to $W^{1,p}(\Omega)$. 
In the vectorial case, we will write $L^p(\Omega;\R^d)\cong L^p(\Omega)^d$ 
and $W^{1,p}(\Omega;\R^d)\cong W^{1,p}(\Omega)^d$. 


Given the fixed time interval $I=[0,T]$, we denote by $L^p(I;X)$ the 
standard Bochner space of Bochner-measurable mappings $I\to X$, where
$X$ is a Banach space. Moreover,  $W^{k,p}(I;X)$ denotes the Banach space of 
mappings in  $L^p(I;X)$ whose $k$-th distributional derivative in time is 
also in $L^p(I;X)$.

Let us list here the  the assumptions on the data  which
are used in the following: 
\begin{subequations}\label{ass}
\begin{align}
&y_0\!\in\! H^2
(\Omega;\R^d),\ \ 
v_0\!\in\! L^2(\Omega;\R^d),\ \ 
\PP_0\!\in\! W^{1,q}
(\Omega;\R^{d\times d}),
\ \ \ q>d,
\\&\alpha_0\!\in\! H^1
(\Omega),\ \ \, 
\phi_0\!\in\! H^1
(\Omega),\ \ \, 
 \zeta_0 \in H^1(\Omega),\ \ \, 
\theta_0\!\in\! L^1(\Omega),\ \ \, 0\le\zeta_0\le1,\ \ \theta_0\ge0,
\label{ass-IC}
\\&\label{ass-load}
g\in C(\R^d;\R^d),\ \ \ 
\mu_\flat\in L^2(\Sigma),\ \  \theta_\flat\in L^1(\Sigma),\ \ \theta_\flat\ge0,
\\&\lambda,\GG,\gamma,m{:}\R^2\to\R^+,\ \chi,\PORO{:}[0,1]\to\R^+\ 
\text{ Lipschitz cont.},\ 
\lambda\ge-\frac2d\GG,\ \ \GG>0,
\label{ass-lambda...m}
\\&
\partial_\alpha\lambda=\partial_\alpha\GG=\partial_\alpha\gamma=\partial_\alpha m=\partial_\alpha\chi=0\ \ \text{
  if $\alpha\not\in[0,1]$},
\\&
\partial_\phi\lambda =\partial_\phi\GG =\partial_\phi\gamma
=\partial_\phi m =\partial_\phi \PORO=0\;\ \ \text{ if
  $\phi\not\in[0,1]$},
\\\label{ass-M-K}&\mathbb M,\,\mathbb K
:\R{\times}\R\to\R^{d\times d}\ \ 
\text{continuous, bounded, and}
\\&\nonumber
\qquad\qquad\qquad\qquad\qquad\text{positive-definite 
(uniformly in their arguments),}
\\&\quad\label{ass-varpi}
\WE{\varpi(a)\ge\begin{cases}\epsilon/a^q&\text{if }a>0,
\\+\infty&\text{if }a\le0,
\end{cases}
\ \ \ \ \ q\ge\frac{pd}{p-d},\ \ \ p>d,\ \ \ \epsilon>0,}
\\&\label{ass-R-monotone}
\mathfrak{R}(\alpha,\phi,\theta;\cdot):\R^{d\times d}\to\R^+\ \text{ and }\
\mathfrak{D}(\alpha,\phi,\theta;\cdot,\cdot):\R^2\to\R^+\text{ convex, 
and }
\\&\label{ass-D-monotone}
\quad\exists a_\mathfrak{R}^{},a_\mathfrak{D}^{}>0\ 
\forall\alpha,\phi,\theta,
\PR_1,\PR_2,\hat\alpha_1,\hat\phi_1,\hat\alpha_2,\hat\phi_2 
:\ \ \ 
\\&\quad\nonumber
(\partial_{\PR}\mathfrak{R}(\alpha,\phi,\theta;\PR_1)
-\partial_{\PR}\mathfrak{R}(\alpha,\phi,\theta;\PR_2)){:}(\PR_1{-}\PR_2)
\ge a_\mathfrak{R}^{}|\PR_1{-}\PR_2|^2,
\\&\nonumber
\quad
\bigg(\partial_{(\hat\alpha,\hat\phi)}\mathfrak{D}\Big(\alpha,\phi,\theta;\binom{\dot\alpha_1}{\dot\phi_1}\Big)
-\partial_{(\hat\alpha,\hat\phi)}\mathfrak{D}\Big(\alpha,\phi,\theta;\binom{\dot\alpha_2}{\dot\phi_2}\Big)\bigg)
{\cdot}\binom{\hat\alpha_1{-}\hat\alpha_2}{\hat\phi_1{-}\hat\phi_2}\\&\qquad
\hspace*{22em}
\ge a_\mathfrak{D}^{}
\left|\binom{\hat\alpha_1{-}\hat\alpha_2}{\hat\phi_1{-}\hat\phi_2}\right|^2
\nonumber
,
\\\label{ass-R}&\quad a_\mathfrak{R} |\PR|^2\le
\mathfrak{R}(\alpha,\phi,\theta;\PR)\le(1+|\PR|^2)/a_\mathfrak{R} ,
\ \ \ 
\\\label{ass-D}&\quad
 a_\mathfrak{D} |\hat\alpha|^2+ a_\mathfrak{D} |\hat\phi|^2\le
\partial_{(\hat\alpha,\hat\phi)}\mathfrak{D}\Big(\alpha,\phi,\theta;\Big(\!\begin{array}{c}\hat\alpha\\[-.1em]\hat\phi\end{array}\!\Big)\Big)
{\cdot}\Big(\!\begin{array}{c}\hat\alpha\\[-.1em]\hat\phi\end{array}\!\Big)
\le(1+|\hat\alpha|^2+|\hat\phi|^2)/ a_\mathfrak{D} ,
\\\label{ass-cv}&c_{\rm v}: \R_+ \to \R_+  \ \text{ continuous, bounded, 
with positive infimum.}\ 
\end{align}\end{subequations} 


Let us mention that assumptions \eqref{ass-R-monotone} and
\eqref{ass-D-monotone}  make sense also for 
 $\mathfrak{R}(\alpha,\phi,\theta;\cdot)$ and 
$\mathfrak{D}(\alpha,\phi,\theta;\cdot)$ nonsmooth. In this case, their 
subdifferentials are indeed set-valued and thus 
\eqref{ass-R-monotone}-\eqref{ass-D} are to be satisfied for any selection 
from these subdifferentials.  We however stick with $\mathfrak{R}$ and
$\mathfrak{D}$ being smooth both for the \REF{sake} of simplicity and 
in accord with models used in geophysical literature 
\cite{HaLyAg04EDPP,HaLyAg05DRDC,LyaHam07DEFF,LyaBeZ08SREA,LyaBeZ09EGMP,LyaBeZ14CDBF}, cf.\ \cite{RouSte??TELS} for details about the treatment of 
the nonsmooth variant of the viscoplasticity.
\WE{An example for $\varpi$ considered already in \cite{RouSte??TELS}
is 
\begin{align*}
\varpi(\det\PP)=\begin{cases}\displaystyle{
\frac{\delta}{\max(1,\det\PP)^q}
+\frac{(\det\PP-1)^2}{2\delta}}\!\!&\text{ if }\ \det\PP>0,\\
\qquad+\infty&\text{ if }\ \det\PP\ge0\,;\end{cases}
\end{align*}
note that the minimum of this potential is attained just at the set 
${\rm SL}(d)$ of the isochoric plastic strains, and that it complies with 
condition \eqref{ass-varpi} for $q\ge pd/(p-d)$
and also with the plastic-indifference condition \eqref{frame-indif}.} 

We are now in the position of making our notion of weak solution precise. 
 
\begin{definition}[Weak formulation of 
\eqref{system}--\eqref{BC}--\eqref{IC}]
\label{def}
We call 
the seven-tuple 
$(y,\PP,\alpha,\phi,\zeta,\por,\theta)$ with
\begin{align*}
  &y\in L^\infty(I;H^2(\Omega;\R^d))\cap H^1(I;L^2(\Omega;\R^d),\\[-.3em]
  &\PP\in L^\infty(I;W^{1,q}(\Omega;\R^{d\times d}))\cap
  H^1(I;L^2(\Omega;\R^{d\times d})),\ \ \WE{\det\PP>0,\ \ \frac1{\det\PP}\in L^\infty(Q),}\\[-.3em]
&\alpha,\phi,\zeta\in
  L^\infty(I;H^1(\Omega))\cap H^1(I;L^2(\Omega)),\\
&\mu\in
  L^2(I;H^1(\Omega)),\ \
\theta
\in L^1(I;W^{1,1}(\Omega))
\end{align*}
a {weak solution} to the initial-boundary-value problem
\eqref{system}--\eqref{BC}--\eqref{IC}
if the following hold:
\begin{itemize}\item[\rm (i)]
The weak formulation of the momentum balance \eqref{momentum-eq} with 
\eqref{elastic-stresses} 
\end{itemize}
\begin{subequations}\label{weak-form}\begin{align}\label{momentum-weak}
&\int_Q\Big( \partial_{F_{\rm el}}^{}\FM(\PORO(\phi)\nabla y\WE{\PP^{-1}}\!,\alpha,\phi,\zeta){:}(\PORO(\phi)\nabla\tilde y\,\WE{\PP^{-1}})
+\kappa_0\nabla^2y{\Vdots}\nabla^2\tilde y
-\varrho\DT y{\cdot}\DT{\tilde y}\Big)\,\d x\d t
\\[-.4em]\nonumber
&\hspace{3em}
%
%
+\int_\Sigma\! N y{\cdot}\tilde y\,\d S\d t=\int_Q\!
g(y){\cdot}\tilde y\,\d x\d t
+
\int_\Omega\!v_0{\cdot}\tilde y(0)\,\d x+\int_\Sigma\!
N y_\flat{\cdot}\tilde y\,\d S\d t
\end{align}
\begin{itemize}\item[]
holds for any $\tilde y$ smooth with $\tilde y(T)=0$.
\item[\rm (ii)] The weak formulation of the plastic flow rule 
\eqref{flow-rule-pi} 
in the form \eqref{flow-rule-pi+} with \eqref{def-of-Sin}
\end{itemize}
\begin{align}\label{weak-form-P}
&\int_Q\Big(
 \PORO(\phi)\nabla y^\top\partial_{F_{\rm el}}^{}\FM
(\PORO(\phi)\nabla y\,\WE{\PP^{-1}}\!,\alpha,\phi,\zeta){:}
( \partial_\PP \WE{\PP^{-1}}{:}\widetilde\PP)
\WE{\,+\varpi'(\det\PP)\WE{\PP^{-T}}{:}\widetilde\PP}
\\[-.4em]&\hspace{3em}\nonumber
+
\partial_\PR^{}\mathfrak{R}\big(\alpha,\phi,\theta;\DT\PP\PP^{-1}\big)\PP^{-\top}
{:}\widetilde\PP
+
\kappa_1|\nabla\PP|^{q-2}\nabla\PP\Vdots\nabla\widetilde\PP
\Big)\,\d x\d t
=0
\nonumber
\end{align}
\begin{itemize}\item[]
holds for any $\widetilde\PP$ smooth.
\item[\rm (iii)] The weak formulation of the coupled flow rule 
\eqref{flow-rule} for $(\alpha,\phi)$
holds for any $\tilde\alpha$ and $\tilde\phi$ smooth:
\end{itemize}
\begin{align}\label{damage-porosity-weak}
&\int_Q\bigg(\partial_{F_{\rm el}}^{}\FM
(\PORO(\phi)\nabla y\,\WE{\PP^{-1}},\alpha,\phi,\zeta){:}(\sigma'(\phi)\nabla y\WE{\PP^{-1}}\tilde\phi)
\\[-.3em]&\nonumber\hspace{.1em}
+
\partial_{(\alpha,\phi)}\FM(\PORO(\phi)\nabla y\,\WE{\PP^{-1}},\alpha,\phi,\zeta)
\cdot\binom{\tilde\alpha}{\WE{\tilde \phi}} 
+\partial_{(\DT\alpha,\DT\phi)}\mathfrak{D}\big(\alpha,\phi,\theta,\PP;\DT\alpha,\DT\phi\big)\cdot\binom{\WE{\tilde\alpha}}{\tilde \phi}  
\\[-.3em]&\hspace{15em}
+\kappa_2\nabla\alpha{\cdot}\nabla\tilde\alpha
+\kappa_3\nabla\phi{\cdot}\nabla\tilde\phi
\bigg)\,\d x\d t=0.
\nonumber
\end{align}
\begin{itemize}
\item[\rm (iv)] The weak formulation of the Cahn-Hilliard 
problem for water-transport equation 
\eqref{system-Darcy}--\eqref{system-Cahn-Hilliard}
\end{itemize}


\begin{align}
&\int_Q\mathscr{M}(\PP,\alpha,\phi)\nabla \por{\cdot}\nabla\tilde\por
-\zeta\DT{\tilde\por}\,\d x\d t+\int_\Sigma M\por\tilde\por\,\d S\d t
\\[-.3em]\nonumber
&\qquad\qquad\qquad\qquad\qquad\qquad=\int_\Omega\zeta_0\tilde\por(0)\,\d x
+\int_\Sigma M\mu_\flat\tilde\por\,\d S\d t
\end{align}
\begin{itemize}\item[]holds for all smooth $\tilde\por$ with $\tilde\por(T)=0$,  $\zeta$  takes values in $[0,1]$ and 
\end{itemize}
\begin{align}\label{mu-weak}
&\int_Q\Big(\FM(\PORO(\phi)\nabla y\,\WE{\PP^{-1}},\alpha,\phi,\tilde\zeta)
-\FM(\PORO(\phi)\nabla y\,\WE{\PP^{-1}},\alpha,\phi,\zeta)-\por(\tilde\zeta-\zeta)
\\[-.3em]&\hspace{3em}
-\kappa_4\nabla\zeta{\cdot}\nabla(\tilde\zeta-\zeta)
+\tauR\DT\zeta\tilde\zeta\Big)\,\d x\d t\nonumber
+\int_\Omega\frac12\tauR\zeta^2_0\,\d x
\ge\int_\Omega\frac12\tauR\zeta^2(T)\,\d x
\nonumber
\end{align}
\begin{itemize}\item[]for all $\tilde\zeta$ smooth valued in $[0,1]$.
\item[\rm (v)] The weak formulation of the heat equation \eqref{system-heat}
\end{itemize}
\begin{align}&\int_Q\!\mathscr{K}(\PP,\phi,\zeta,\theta)\nabla\theta{\cdot}
\nabla\tilde\theta-C_{\rm v}(\theta)\DT{\tilde\theta}
-r\tilde\theta\,\d x\d t+\int_\Sigma\! K\theta\tilde\theta\,\d S\d t 
\\[-.3em]&\qquad\qquad\qquad\qquad\qquad
=\int_\Sigma\! K\theta_\flat\tilde\theta\,\d S\d t
+\int_\Omega\! C_{\rm v}(\theta_0)\tilde\theta(0)\,\d x
\nonumber\end{align}\end{subequations}
\begin{itemize}\item[]holds for any $\tilde\theta$ smooth with $\tilde\theta(T)=0$
and with $C_{\rm v}(\cdot)$ denoting a primitive function to $c_{\rm
  v}(\cdot)$ and with
$r=r(\PP,\alpha,\phi,\theta;\DT\PP,\DT\alpha,\DT\phi,\DT\zeta,\nabla
\por^{})$ from \eqref{system-diss}.
\item[\rm (vi)] The remaining initial conditions  $y(0)=y_0$,   $\PP(0) = \PP_0$,  $\alpha(0)=\alpha_0$,
and $\phi(0)=\phi_0$ are satisfied.
\end{itemize}
\end{definition}



 Our main analytical result is an existence theorem for weak
solutions.  This is to  be seen as a mathematical consistency property of
the proposed model. It reads as follows. 

\begin{theorem}[Existence of weak solutions]\label{thm}
Let the assumptions \eqref{ass} hold. Then,  there exists  a weak solution 
$(y,\PP,\alpha,\phi,\zeta,\por,\theta)$  in the sense of 
Definition~{\rm {\ref{def}}}.  In addition 
\begin{align}\label{Laplaceans}
{\rm div}\big(|\nabla\PP|^{q-2}\nabla\PP\big)\in L^2(Q;\R^{d\times d}),\ \ \ \ \ \ 
\Delta\alpha\in L^2(Q),\ \ \text{ and }\ \ \Delta\phi\in L^2(Q). 
\end{align}
Moreover, the energy conservation \eqref{energy-conserv+} holds on the time 
intervals $[0,t]$ for all $t\in I$ in the  following  sense with 
$\vartheta(t)=C_{\rm v}(\theta(t))$:
\begin{align}\label{energy-conserv++}
& \int_\Omega\frac\varrho2|\DT y(t)|^2
+\vartheta(t)\,\d x
+\PsiM(\nabla y(t),\PP(t),\alpha(t),\phi(t),\zeta(t))
+\int_\Gamma \frac12N|y(t)|^2\d S
\\[-.2em]\nonumber
&\quad =\int_0^t\!\!\int_\Omega g(y)\cdot\DT y\,\d x\d t
+\int_0^t\!\!\int_\Gamma Ny_\flat\cdot\DT y+M(\por^{}{-}\mu_\flat)
+K(\theta{-}\theta_\flat)\,\d S\d t
\\[-.2em]&\qquad\quad+
\int_\Omega\frac\varrho2|v_0|^2+C_{\rm v}(\theta_0)\,\d x
+\PsiM(\nabla y_0,\PP_0,\alpha_0,\phi_0,\zeta_0)
+\int_\Gamma \frac12N|y_0|^2\d S\,.
\nonumber\end{align}
\end{theorem}

We will prove this  result  in Propositions~\ref{prop-est}--\ref{prop-conv2}
by a suitable regularization, transformation, and approximation  procedure. 
This  also provides a (conceptual) algorithm that is numerically stable  and 
converges as the discretization and the regularization parameters $h>0$ and 
$\varepsilon>0$ tend to $0$\WE{.} (More  specifically, when successively 
$h\to0$ and then $\varepsilon\to0$ and when $\theta$ is reconstructed from the 
rescaled temperature $\vartheta$ used below through \eqref{def-of-vartheta}).

\section{Convergence of Galerkin approximations}\label{sec-proof}
 We devote section to the proof of the existence result,  namely
 Theorem~\ref{thm}. As already mentioned,  
we apply a constructive method  delivering an approximation of
 the problem.  This results from combining a regularization in
 terms of the small parameter $\varepsilon$ and a Galerkin
 approximation, described by the small parameter $h>0$ instead.  In
 particular, we prove  the existence of approximated solutions,
 their stability (a-priori
estimates), and their  convergence  to   weak solutions, at least in terms of subsequences.  
The general philosophy of   a-priori  estimation  relies on 
 the fact that  temperature  plays a role in connection
with dissipative
mechanisms only: adiabatic 
effects are omitted and most estimates on the mechanical part of the system are
independent of temperature and its discretization. In addition to
this,  the viscous nature of the Cahn-Hilliard 
model (\ref{system}d,e)  allows us to obtain  useful estimates
 even in absence of additional  Kelvin-Voigt-type
viscosity,   which otherwise would bring 
additional mathematical complications.
The estimates and the convergence
rely on the  independence of the heat capacity of mechanical 
variables.  Let us however note that additional dependencies in
$c_{\rm v}$  could
be considered along the lines of  \cite{Roub13NCTV,Roub13NPDE}. 


 Let us begin by detailing the regularization. This 
  concerns the heat-production rate $r$ from 
\eqref{system-diss} as well as the prescribed heat flux on the boundary and the 
initial condition. More specifically, for  some given  regularization parameter 
$\varepsilon>0$, we replace  these terms  respectively by 
\begin{subequations}\label{regularization} 
\begin{align}\label{regularization-r} 
&r_\varepsilon=\frac{r\big(\PP,\alpha,\phi,\theta;\DT\PP,\DT\alpha,\DT\phi,\DT\zeta,\nabla\por\big)}
{1+\varepsilon\big(|\DT\PP^{}\PP^{-1}|^2+|\DT\alpha|^2+|\DT\phi|^2+|\nabla\por|^2\big)}\,,
\\&\label{regularization-BC-IC} 
\theta_{\flat\varepsilon}=\frac{\theta_\flat}{1+\varepsilon\theta_\flat},\ \ \ \ \ \ \ \text{ and }\ \ \ \ \ \ 
\theta_{0\varepsilon}=\frac{\theta_0}{1+\varepsilon\theta_0}.
\end{align}\end{subequations}
Due to the boundedness/growth assumptions (\ref{ass}g,j,k), the dissipation
rate $r$ has a quadratic growth in rates and thus
$r_\varepsilon$ is bounded  as well as 
$ \theta_{\flat,\varepsilon}  $ and
$\theta_{0\varepsilon}$. As effect of this boundedness, we are in the
position of resorting to a $L^2$-theory instead of the $L^1$-theory for the
regularized heat problem.  In addition, we perform a
regularization of the nonsmooth term $\mathfrak{N}_{[0,1]}^{}$ in
\eqref{p-por} by means of its  Yosida approximation, \WE{yielding}
the  mapping $\mathfrak{N}_\varepsilon$ \WE{defined as}

\vspace*{-2em}
\begin{align}\label{regularization-of-N} 
\mathfrak{N}_\varepsilon(\zeta)=\begin{cases}
\zeta/\varepsilon&\text{if }\zeta<0,\\[-.2em]
0&\text{if }0\le\zeta\le1,\\[-.2em]
(\zeta{-}1)/\varepsilon&\text{if }\zeta>1.
\end{cases}
\end{align}

In order to simplify the convergence proof, we apply the so-called 
enthalpy transformation  to the heat equation. This consists in rescaling 
temperature by introducing a new variable 
\begin{align}\label{def-of-vartheta}
\vartheta=C_{\rm v}(\theta) 
\end{align}
 where, we recall, $C_{\rm v}$ is the primitive of $c_{\rm v}$
vanishing in $0$. 
Note that $\DT\vartheta=c_{\rm v}(\theta)\DT\theta$ and that
$C_{\rm v}$ is increasing so that its inverse  $ C_{\rm v}^{-1}$ exists 
and $\nabla\theta=\nabla C_{\rm v}^{-1}(\vartheta)=\nabla\vartheta/c_{\rm v}(\theta)
=\nabla\vartheta/c_{\rm v}(C_{\rm v}^{-1}(\vartheta))$.  Upon
letting 
\begin{align*}
\mathfrak{K}(\PP,\phi,\zeta,\vartheta):=
\frac{1}{{c_{\rm v}(C_{\rm v}^{-1}(\vartheta))}}\mathscr{K}(\PP,\phi,\zeta,C_{\rm v}^{-1}(\vartheta)),\ \ \ \ 
\end{align*}
 we rewrite and regularize the system \eqref{system} by 
\begin{subequations}\label{system+}
\begin{align}\label{momentum-eq+}
&\varrho\DDT y={\rm div}
\,
\varSigma_{\rm el}
+
g(y),
\\&\label{flow-rule-pi++}
\partial_\PR^{}\mathfrak{R}\big(\alpha,\phi,C_{\rm v}^{-1}(\vartheta);
\DT\PP\PP^{-1}\big)
+
\,\varSigma_{\rm in}\PP^\top
=0,
\\[-.1em]&
\partial_{(\DT\alpha,\DT\phi)}\mathfrak{D}\Big(\!\alpha,\phi,C_{\rm v}^{-1}(\vartheta);
\Big(\!\begin{array}{c}\DT\alpha\\[-.3em]\DT\phi\end{array}\!\Big)\Big)
+\Big(\!\!\begin{array}{c}
p_{\rm age}\!
\\[-.1em]
p_{\rm eff}\,
\end{array}\!\!\Big)
=\Big(\!\!\begin{array}{c}0\!\!
\\[-.1em]
0\!\!\end{array}\Big),
\label{flow-rule+}
\\[-.3em]\label{system-Darcy+}
&\DT\zeta=\EE{\rm div}\big(\mathscr{M}(\PP,\alpha,\phi)
\nabla \por^{}\big)
,
\\[-.3em]\label{system-Cahn-Hilliard+}
&\por=
 \partial_\zeta \FM  -\kappa_4\Delta\zeta+\mathfrak{N}_\varepsilon(\zeta)+\tauR\DT\zeta,
\\
&\DT\vartheta
={\rm div}\big(\mathfrak{K}(\PP,\phi,\zeta,\vartheta)\nabla\vartheta\big)
+\frac{r\big(\PP,\alpha,\phi,C_{\rm v}^{-1}(\vartheta);\DT\PP,\DT\alpha,\DT\phi,\DT\zeta,\nabla \por^{}\big)}
{1+\varepsilon\big(|\DT\PP^{}\PP^{-1}|^2\!+|\DT\alpha|^2\!+|\DT\phi|^2\!+|\nabla\por|^2\big)}
\!\!
\label{system-heat+}
\end{align}
\end{subequations}
where $\varSigma_{\rm el}$ and $\varSigma_{\rm in}$ are again from 
(\ref{stresses}a,b) and $r$ is again given in \eqref{system-diss} and
$\mathfrak{N}_\varepsilon$ is  defined in 
\eqref{regularization-of-N}.  Note that the
$\varepsilon$-regularization serves the double purpose of having a
bounded right-hand side in \eqref{system-heat+} as well as a smooth
nonlinearity $\mathfrak{N}_\varepsilon$ in
\eqref{system-Cahn-Hilliard+}.  The boundary conditions 
are correspondingly modified by using \eqref{regularization-BC-IC}, i.e.\ 
$\mathscr K(\PP,\zeta,\theta)\nabla\theta{\cdot}\nu+K\theta=K\theta_\flat(t)$
in \eqref{BC-2} and $\theta(0)=\theta_0$ in \eqref{IC} modify respectively 
as
\begin{align}\label{BC-IC+}
\mathfrak{K}(\PP,\phi,\zeta,\vartheta)\nabla  C_{\rm
  v}^{-1}(\vartheta)  {\cdot}\nu
+K C_{\rm v}^{-1}(\vartheta)=K\theta_{\flat\varepsilon}(t), \ \ \ \ 
\vartheta(0)=\vartheta_{0\varepsilon}:=C_{\rm v}(\theta_{0\varepsilon})
\end{align}
with $\theta_{\flat\varepsilon}$ and $\theta_{0\varepsilon}$ from 
\eqref{regularization-BC-IC}.

A possible way of approximating \eqref{system+} is via a discretisation in 
time (sometimes, in its backward-Euler variant, called the {\it Rothe
method}). This would however give rise to mathematical difficulties because 
of the remarkable nonconvexity of the model, making estimation 
and even existence of discrete solutions troublesome.  
Note in particular that the (generalized) St.Venant-Kirchhoff ansatz 
\eqref{ansatz+}, which we have in mind as a prominent example, is 
already severely nonconvex (and even not semi-convex).

We therefore resort \REF{to} using a Galerkin  approximation in space instead 
(which, in its evolution variant,  is  sometimes referred to as 
{\it Faedo-Galerkin method}). For possible numerical implementation, one can 
imagine a conformal finite element  formulation,  with $h>0$ denoting
the {\it mesh size}.  Assume for simplicity that the sequence of
nested finite-dimensional subspaces $V_h \subset H^2(\Omega)$ invading
$H^1(\Omega)$ are given. We shall use these spaces for all scalar
variables (i.e., $\alpha$, $\phi$, $\zeta$, $\mu$, and $\vartheta$) so
that Laplacians are defined in the usual strong sense.   This  will allow
some simplification in the estimates. 
It is also important to choose 
the same sequences of finite-dimensional 
subspaces for both \eqref{system-Darcy+} and
\eqref{system-Cahn-Hilliard+}  in order to 
to facilitate cross-testing and  the  cancellation 
of the terms $\pm\por\DT\zeta$ also on the Galerkin-approximation level.
For simplicity,  we  assume that all initial conditions 
$(y_0,\PP_0,\alpha_0,\phi_0,\zeta_0,\vartheta_{0\varepsilon})$ 
belong to  
all finite-dimensional subspaces so that  no additional
approximation of such conditions is needed.


 The outcome of the Galerkin approximation is an 
 an initial-value problem for a system of ordinary 
differential-algebraic equations. The algebraic constraint arises from 
\eqref{system-Darcy+} and \eqref{system-Cahn-Hilliard+} by eliminating 
$\DT\zeta$, i.e.
\begin{align}\label{alg-const}
\por=
 \partial_\zeta \FM  -\kappa_4\Delta\zeta+\mathfrak{N}_\varepsilon(\zeta)+\tauR{\rm div}\big(\mathscr{M}(\PP,\alpha,\phi)
\nabla \por^{}\big).
\end{align}
In \eqref{est-w} below, we denote $|\cdot|_{h}^*$  the  seminorm on 
$L^2(I;H^1(\Omega)^*)$ defined by
\begin{align}\label{seminorm}
|\xi|_{h}^*:=\sup \left\{\int_Q\xi v\,\d x\d t \ :\
\|v\|_{L^2(I;H^1(\Omega))}\le1,\ v(t)\in V_{h} \ \ \text{for a.e.} \
t \in I\right\}. 
\end{align}
Similar seminorms  (with the same notation) are defined on 
spaces tensor-valued functions. On $L^2$-spaces we let  
\begin{align}\label{seminorm+}
|\xi|_{h}:=\sup \left\{\int_Q\xi v\,\d x\d t \ :\
\|v\|_{L^2(Q)}\le1,\ v(t)\in V_{h} \ \ \text{for a.e.} \
t \in I\right\}, 
\end{align}
to be used for \eqref{est-Delta-Pi} and \eqref{est-Delta-phi}
below.  This family of these seminorms make  
the linear spaces $L^2(I;H^1(\Omega)^*)$ 
and $L^2(Q;\R^{d\times d})$ and $L^2(Q)$ 
metrizable locally convex spaces (Fr\'echet spaces). 
\def\eps{\varepsilon}
 Henceforth, we use the symbol $C$ to indicate a
positive constant, possibly depending on data but independent from
regularization and discretization parameters. Dependences on such
parameters will be indicated in indices. Our stability result
reads as follows. 

\begin{proposition}[Discrete solution
and a priori  estimates]\label{prop-est}
Let  assumptions  \eqref{ass} hold and $\eps,h>0$ be fixed. Then, the Galerkin
approximation of \eqref{system+} with the initial/boundary conditions
\eqref{BC}---\eqref{IC} modified by \eqref{BC-IC+}  admits a solution 
on the whole time
interval $I=[0,T]$, let us denote it by $(y_{\eps h},\PP_{\eps h},\alpha_{\eps h},\phi_{\eps h},\zeta_{\eps h},\por_{\eps h},\vartheta_{\eps h})$, \WE{such that $\PP_{\eps h}$
is invertible} and 
we have the estimates 
\begin{subequations}\label{est}\begin{align}
&\big\|y_{\eps h}\big\|_{L^\infty(I;H^2
(\Omega;\R^d))\,\cap\,W^{1,\infty}(I;L^2(\Omega;\R^d))}^{}\le C,
\\[-.3em]\label{est-P}
&\big\|\PP_{\eps h}\big\|_{L^\infty(I;W^{1,q}(\Omega;\R^{d\times d}))\,\cap\,H^1(I;L^
2(\Omega;\R^{d\times d}))}^{}
\le C
\WE{\ \text{ and }\ \Big\|\frac1{\det\PP_{\eps h}}\Big\|_{L^\infty(Q)}\le C},
\\[-.3em]&\big\|\alpha_{\eps h}\big\|_{L^\infty(I;H^1
(\Omega))\,\cap\,H^1(I;L^2(\Omega))}^{}\le C,
\\&\big\|\phi_{\eps h}\big\|_{L^\infty(I;H^1
(\Omega))\,\cap\,H^1(I;L^2(\Omega))}^{}\le C,
\\&\big\|\zeta_{\eps h}\big\|_{L^\infty(I;H^1(\Omega))\,\cap\,H^1(I;L^2(\Omega))}^{}\le C,
\\&\label{est-por-H1}
\big\|\por_{\eps h}\big\|_{L^2(I;
H^1(\Omega))}^{}\le C,
\\&\label{est-theta-H1}
\big\|\vartheta_{\eps h}\big\|_{L^2(I;H^1
(\Omega))}\le C_\eps,
\\&\label{est-w}
\big|\DT\vartheta_{\eps h}\big|_{h_0}^*\le C_\eps\ \ \ \ \text{ for }\ h_0\ge h>0,
\\&
\big|{\rm div}(|\nabla\PP_{\eps h}|^{q-2}\nabla\PP_{\eps h})\big|_{h_0}^{}\le C
\ \ \ \ \text{ for }\ h_0\ge h>0,
\label{est-Delta-Pi}
\\&
\big|\Delta\alpha_{\eps h}\big|_{h_0}^{}\le C
\quad\text{ and }\quad
\big|\Delta\phi_{\eps h}\big|_{h_0}^{}\le C\ \ \ \ \text{ for }\ h_0\ge h>0,
\label{est-Delta-phi}
\\&\label{est-of-penalization-}
\big\|\min(0,\zeta_{\eps h}\big)\|_{L^\infty(I;L^2(\Omega))}\le
C/\sqrt\eps, \ 
\big\|\max(1,\zeta_{\eps h}\big)\|_{L^\infty(I;L^2(\Omega))}\le C/\sqrt\eps.
\end{align}\end{subequations}
\end{proposition}

\begin{proof}[Sketch of the proof]
The existence of  a global solution to the   Galerkin
approximation   follows directly by  the usual
successive-continuation argument.  The algebraic constraint \eqref{alg-const} 
for 
the underlying system of ordinary differential-algebraic equations
takes  the   more specific form
\begin{align*}
\por=
m(\alpha,\phi)\frac{\zeta{-}\phi{-}\beta I_1}{\sqrt[4]{1{+}\epsilon I_2}}-\kappa_4\Delta\zeta
+\mathfrak{N}_\varepsilon(\zeta)+\tauR{\rm div}\big(\mathscr{M}(\PP,\alpha,\phi)
\nabla \por^{}\big).
\end{align*} 
The matrix arising  by approximating   the linear operator 
$\por\mapsto\por-\tauR{\rm div}(\mathscr{M}(\PP,\alpha,\phi)\nabla\por)$
 along   with the linear boundary condition \eqref{BC-2} 
turns out to be  positive definite,  therefore invertible.  Thus, we can 
obtain a solution to the underlying system 
of ordinary-differential equations,  for the 
differential-algebraic system has index 1. 

Let us now move to \WE{the} a-priori estimation. We start by recovering the
mechanical energy balance, see \eqref{energy-conserv} with
\eqref{system-diss}. In particular,  we use 
$\DT y_{\eps h}$, $\DT\PP_{\eps h}$, $\DT\alpha_{\eps h}$, $\DT\phi_{\eps h}$, 
$\por_{\eps h}$, and $\DT\zeta_{\eps h}$ as test functions  into
each corresponding equation discretized by the Galerkin method. All these 
tests are legitimate, provided the finite-dimensional spaces used in both 
equations in the Cahn-Hilliard systems (\ref{system+}b,e) are the same so the 
terms $\pm\por_{\eps h}\DT\zeta_{\eps h}$ cancel out even in the discrete 
level.  More specifically, using $\DT y_{\eps h}$  as test in the
Galerkin approximation of  \eqref{momentum-eq+} with 
its boundary condition \eqref{BC-1}, we obtain
\begin{align}\label{test-of-moment}
&\int_\Omega\!\frac\varrho2|\DT y_{\eps h}(t)|^2
+\frac{\kappa_0}2|\nabla^2y_{\eps h}(t)|^2\,\d x
+
\int_\Gamma\frac12N|y_{\eps h}(t)|^2\,\d S
\\\nonumber&\qquad+
\int_0^t\!\!\int_\Omega\! \partial_{ \nabla y } \widehatFM(\nabla y_{\eps h},\PP_{\eps h},\alpha_{\eps h},\phi_{\eps h},\zeta_{\eps h}){:}\nabla\DT y_{\eps h}\,\d x\d t
=\int_0^t\!\!\int_\Omega\!g(y){\cdot}\DT y_{\eps h}\,\d x\d t
\\&\qquad+\int_0^t\!\!\int_\Gamma Ny_\flat{\cdot}\DT y_{\eps h}\,\d S\d t
+\int_\Omega\!\frac\varrho2|v_0|^2+\frac{\kappa_0}2|\nabla^2y_0|^2\,\d x
+\int_\Gamma\frac12N|y_0|^2\,\d S.
\nonumber
\end{align}
 By testing the Galerkin approximation of 
\eqref{flow-rule-pi++}  
by $\DT\PP_{\eps h}$  one gets 
\begin{align*}\nonumber
&\int_\Omega\!\frac{\kappa_1}q|\nabla\PP_{\eps h}(t)|^q\d x
+\int_0^t\!\!\int_\Omega\!\partial_{\PR}{\mathfrak{R}}\big(\alpha_{\eps h},\phi_{\eps h},\theta_{\eps h};\DT\PP_{\eps h}\PP_{\eps h}^{-1}\big)
{:}\DT\PP_{\eps h}\PP_{\eps h}^{-1}
\\[-.3em]
&\qquad+
 \partial_\PP  \widehatFM(\nabla y_{\eps h},\PP_{\eps h},\alpha_{\eps h},\phi_{\eps h},\zeta_{\eps h})
{:}\DT\PP_{\eps h}\d x\d t
=\int_\Omega\!\frac{\kappa_1}q|\nabla\PP_0|^q\d x.
\end{align*}
Next, we test the Galerkin approximation of \eqref{flow-rule+} by 
$(\DT\alpha_{\eps h},\DT\phi_{\eps h})$, which gives
\begin{align}\label{test-of-flow-rule++}
&\int_\Omega\!\frac{\kappa_2}2|\nabla\alpha(t)|^2\d x +
\frac{\kappa_3}2|\nabla\phi(t)|^2\d x
+
\int_0^t\!\!\int_\Omega
\partial_{(\DT\alpha,\DT\phi)}\mathfrak{D}\Big(\alpha_{\eps h},\phi_{\eps h},\theta_{\eps h};\Big(\!\begin{array}{c}\DT\alpha_{\eps h}
\\[-.1em]\DT\phi_{\eps h}\end{array}\!\Big)\Big)
{\cdot}\Big(\!\begin{array}{c}\DT\alpha_{\eps h}\\[-.1em]\DT\phi_{\eps h}\end{array}\!\Big) \Big)\d x\d t\\[-.1em]&\nonumber\qquad
+
\int_0^t\!\!\int_\Omega
\partial_{\alpha}\widehatFM(\nabla y_{\eps h},\PP_{\eps h},\alpha_{\eps h},
\phi_{\eps h},\zeta_{\eps h})\DT\alpha_{\eps h}\d x\d t
\\[-.1em]&
\qquad
+
\int_0^t\!\!\int_\Omega
\partial_{\phi}\widehatFM(\nabla y_{\eps h},\PP_{\eps h},\alpha_{\eps h},\phi_{\eps h},\zeta_{\eps h})\DT\phi_{\eps h}
\d x\d t
=
\int_\Omega\!\frac{\kappa_2}2|\nabla\alpha_0|^2+
\frac{\kappa_3}2|\nabla\phi_0|^2\d x\,.
\nonumber
\end{align}
We now test the Galerkin approximation of \eqref{system-Cahn-Hilliard+} by 
$\DT\zeta_{\eps h}$. Such procedure leads to a (system of ordinary) differential 
equation instead of the inclusion, so that conventional calculus applies. 
This gives 
\begin{align}\label{test-of-Cahn-Hilliard}
&\int_0^t\!\!\int_\Omega\!\tauR\DT\zeta_{\eps h}^2-\por_{\eps h}\DT\zeta_{\eps h}\,\d x\d t
=-\int_0^t\!\!\int_\Omega\!
\partial_\zeta^{} \widehatFM(\nabla y_{\eps h}, \PP_{\eps h},  \alpha_{\eps h},\phi_{\eps h},\zeta_{\eps h})\DT\zeta_{\eps h}
\\[-.4em]&\qquad\qquad\qquad\qquad\qquad\qquad\qquad\qquad
+\kappa_4\nabla\zeta_{\eps h}{\cdot}\nabla\DT\zeta_{\eps h}
+\mathfrak{N}_\varepsilon(\zeta_{\eps h})\DT\zeta_{\eps h}\,\d x\d t
\nonumber
\end{align}
with $\mathfrak{N}_\varepsilon$ from \eqref{regularization-of-N}. 
Testing the Galerkin approximation of \eqref{system-Darcy+} by
$\por_{\eps h}$
we  obtain 
\begin{align}\label{test-of-Darcy}
&\int_0^t\!\!\int_\Omega\!
\mathscr{M}(\PP_{\eps h},\alpha_{\eps h},\phi_{\eps h})\nabla \por_{\eps h}
{\cdot}\nabla \por_{\eps h}
\,\d x\d t+\int_0^t\!\!\int_\Sigma \!M\por_{\eps h}^2\,\d S\d t
\\[-.4em]&\qquad\qquad\qquad\qquad\qquad
=\int_0^t\!\!\int_\Gamma \!
M\por_{\eps h}\mu_\flat
\,\d S\d t
-\int_0^t\!\!\int_\Omega\!\DT\zeta_{\eps h}\por_{\eps h}
\,\d x\d t.
\nonumber
\end{align}
%
%
Summing \eqref{test-of-Cahn-Hilliard} and \eqref{test-of-Darcy} up and
 exploiting the  cancellation of the terms $\pm\int_\Omega
\DT\zeta_{\eps h} \por_{\eps h}\,\d x$, we obtain
\begin{align}\label{C-H-tested}
&
\int_0^t\!\!\int_\Omega\partial_\zeta\widehatFM(\nabla y_{\eps h},\PP_{\eps h},\alpha_{\eps h},\phi_{\eps h},\zeta_{\eps h})\DT\zeta_{\eps h}\,\d x\d t
+\int_0^t\!\!\int_\Gamma M\por_{\eps h}^2\,\d S\d t
\\&\quad \nonumber
+\int_\Omega\frac{\kappa_4}2|\nabla\zeta_{\eps h}(t)|^2\,\d x
+\int_0^t\!\!\int_\Omega\mathscr{M}(\PP_{\eps h},\alpha_{\eps h},\phi_{\eps h})
\nabla\por_{\eps h}{\cdot}\nabla\por_{\eps h}
+\tauR\DT\zeta_{\eps h}^2\,\d x\d t\nonumber
\\&\quad\le\int_\Omega\frac{\kappa_4}2|\nabla\zeta_0|^2\,\d t\d x
+\int_0^t\!\!\int_\Gamma M\por_{\eps h}\mu_\flat\,\d S\d t.
\nonumber
\end{align}
The inequality sign in \eqref{C-H-tested} comes from the fact that 
$\mathfrak{N}_\varepsilon(\zeta_{\eps h})\DT\zeta_{\eps h}\,\d x\d t\le0$, using 
$0\le\zeta_0\le1$  as well,  cf.\ \eqref{ass-IC}.

Taking the sum of  
\eqref{test-of-moment}--\eqref{test-of-flow-rule++}
 and \eqref{C-H-tested}  and using the calculus 
\begin{align}\nonumber
&  \partial_{\nabla y} \widehatFM{:}\nabla\DT y_{\eps h}+\partial_{\PP} \widehatFM{:}\DT\PP_{\eps h}
+
\partial_\alpha\widehatFM\DT\alpha_{\eps h}+
\partial_\phi\widehatFM\DT\phi_{\eps h}
+\partial_\zeta\widehatFM\DT\zeta_{\eps h}
=\frac{\partial}{\partial t}\widehatFM(\nabla y_{\eps h},\PP_{\eps
  h},\alpha_{\eps h},\phi_{\eps h},\zeta_{\eps h}),\nonumber 
\end{align} 
we obtain the discrete analogue of \eqref{energy-conserv}. 

The boundary term in \eqref{test-of-moment} contains $\DT y$, which is not
well defined on $\Gamma$.  We overcome this obstruction by  
by-part integration  
\begin{align}\label{by-part-boundary}
&\int_0^t\!\int_\Gamma Ny_\flat{\cdot}\DT y_{\eps h}\,\d S\d t
=\int_\Gamma Ny_\flat(t){\cdot}y_{\eps h}(t)\,\d S
-\int_0^t\!\int_\Gamma N\DT y_\flat{\cdot}y_{\eps h}\,\d S\d t
-\int_\Gamma Ny_\flat(0){\cdot}y_0\,\d S 
\end{align}
so that this boundary term can be estimated by using the assumption 
\eqref{ass-load} on $y_\flat$. Furthermore, the last term in \eqref{C-H-tested} 
can be estimated as 
\begin{align}
\int_0^t\!\!\int_\Gamma M\mu_\flat\por_{\eps h}\,\d S\d t
\le M C\|\mu_\flat\|_{L^2(\Sigma)}^{}\big(\|\por_{\eps h}\|_{L^2(\Sigma)}^{}+
\|\nabla\por_{\eps h}\|_{L^2(Q;\R^d)}^{}\big)\nonumber
\end{align}
 where $C$ is here  the norm of the trace  operator
$H^1(\Omega)\to L^2(\Gamma)$ (by 
considering the norm $\|\nabla\cdot\|_{L^2(\Omega;\R^d)}+\|\cdot\|_{L^2(\Gamma)}$ on 
$H^1(\Omega)$).

 These estimates allow us to obtain the bounds   (\ref{est}a-f). More
 in detail, \eqref{est-P}  follows from   the 
coercivity \eqref{ass-R} of $\mathfrak{R}$ so that we have also 
that   $\DT\PP_{\eps h}^{_{}}\PP_{\eps h}^{-1}$  is bounded  in
$L^2(Q;\R^{d\times d})$.  In particular, we have here used the
boundary condition on the plastic strain \eqref{BC-3}.

\WE{
An important ingredient was that, exploiting \eqref{ass-varpi}, we can use the 
Healey-Kr\"omer Theorem  \cite[Thm. 3.1]{HeaKro09IWSS}, 
originally devised for the deformation gradient,  as done 
already in \cite{RouSte??TELS} for the plastic strain. 
This gives the second estimate in \eqref{est-P}, which holds at the
Galerkin level as well, so that
in fact the singularity of $\varpi$ is not seen during the evolution and
the Lavrentiev phenomenon is excluded. 
Let us point out that, in the frame of our weak thermal coupling the 
assumption \eqref{ass-R}, these estimates  hold  independently of temperature, 
and thus the constants in (\ref{est}a,b) are independent of $\varepsilon$.

Using the boundedness of the $\mathscr{K}$-term and
the positive definiteness of $\mathbb{K}$ in \eqref{ass-M-K}, and 
recalling \eqref{M-pull-back+}, we get the bound 
$\| \WE{\PP^{-T}_{\eps h}}\nabla\por_{\eps h}/\sqrt{\det\PP_{\eps h}}\|_{L^2(Q)^d}\le C$. 
Then the estimate \eqref{est-por-H1} follows by using
\begin{align}\label{estimation-of-mu}
\|\nabla\por_{\eps h}\|_{L^2(Q)^d}
&=\Big\|\frac{\PP_{\eps h}^{\top}\WE{\PP^{-T}_{\eps h}}}{\det\PP_{\eps h}}\nabla\por_{\eps h}\Big\|_{L^2(Q)^d}
\\&\nonumber
\le\Big\|\frac{\PP_{\eps h}^{}}{\sqrt{\det\PP_{\eps h}}}\Big\|_{L^\infty(Q)^{d\times d}}
\Big\|\frac{\WE{\PP^{-T}_{\eps h}}}{\sqrt{\det\PP_{\eps h}}}\nabla\por_{\eps h}\Big\|_{L^2(Q)^d}\le C,
\end{align}
where the latter bound follows from \eqref{est-P}. 
}

Let us point out that, in the frame of assumptions (\ref{ass}j,k), these 
estimates  hold  independently of temperature, and thus the constants
in (\ref{est}a-f) are independent of $\varepsilon$.


Let us now test the Galerkin approximation of  the heat equation 
\eqref{system-heat+} by $\vartheta_{\eps h}$. 
This test  is allowed at the level of Galerkin approximation, 
although it does not lead to  the total energy balance. 
 We  obtain
\begin{align}\label{heat-tested}
&\frac{\d}{\d t}\frac12\int_\Omega\!\vartheta_{\eps h}^2
\,\d x
+\int_\Omega\!
\mathfrak{K}(\PP_{\eps h},\phi_{\eps h},\zeta_{\eps h},\vartheta_{\eps h})
\nabla\vartheta_{\eps h}{\cdot}\nabla\vartheta_{\eps h}\,\d x
+\int_\Gamma\! K\vartheta_{\eps h}^2\,\d S
\\[-.4em]&
\qquad\qquad\qquad\qquad\qquad\qquad\qquad\qquad =
\int_\Omega\! r_\eps\vartheta_{\eps h}\,\d x+\int_\Gamma\!
K\theta_{\flat\eps}\vartheta_{\eps h}\,\d S.
\nonumber
\end{align}
After integration over $[0,t]$, we use the Gronwall inequality 
and exploit 
the control of the initial condition 
$|\theta_{0\eps}|\le1/\eps$ due to \eqref{regularization-BC-IC}. 
The last boundary term in \eqref{heat-tested}  can be controlled as  
$|\theta_{\flat,\eps}|\le1/\eps$, again due to \eqref{regularization-BC-IC}. 
 By arguing as for the $\mathscr{M}$-term, we use the   
$\mathscr{K}$-term  in order to get  the bound 
$\|\WE{\PP^{-1}_{\eps h}} \nabla\theta_{\eps h}\|_{L^2(Q;\R^d)}\le
C_\eps$ and  then  $\|\nabla\theta_{\eps h}\|_{L^2(Q;\R^d)}\le
C_\eps$, see  \eqref{est-theta-H1}.
\WE{Analogous arguments as \eqref{estimation-of-mu}
lead to the estimate \eqref{est-theta-H1} for 
$\nabla\vartheta_{\eps h}$, now depending on the regularization parameter $\eps$.
}

By comparison, we obtain the estimate \eqref{est-w} of $\DT\vartheta_{\eps h}$
in the seminorm \eqref{seminorm}.  Again  by comparison, using \eqref{flow-rule-pi++} 
with \eqref{def-of-Sin}  and taking advantage of  the boundedness of the term
$\partial_\PR^{}\mathfrak{R}(\alpha_{\eps h},\phi_{\eps h},\theta_{\eps
  h};\DT\PP_{\eps h}\PP_{\eps h}^{-1})\PP_{\eps h}^{-\top}$ in $
L^2(Q;\R^{d\times d})$, the first term in 
\eqref{def-of-Sin}, i.e.\  
$\PORO(\phi_{\eps h})\nabla y_{\eps h}^\top
\partial_{F_{\rm el}}^{}\FM(\cdot){:} \partial_\PP \WE{\PP^{-1}}$, turns
out to be  bounded in $L^2(Q;\R^{d\times d})$, because
$\nabla y_{\eps h}^\top$ is bounded in
$L^\infty(I;L^6(\Omega;\R^{d\times d}))$
and $\partial_{F_{\rm el}}^{}\FM$ is bounded in
$L^2(I;L^3(\Omega;\R^{d\times d}))$ for $d\le3$ and $\PORO(\phi_{\eps h}) \partial_\PP \WE{\PP^{-1}}$ is controlled in 
$L^\infty(Q;\R^{ d\times d\times d\times d})$.   Here, we  emphasize that 
one cannot perform \eqref{flow-rule-pi}  the nonlinear test by 
${\rm div}(|\nabla\PP_{\eps h}|^{q-2}\nabla\PP_{\eps h})$ to obtain the
estimate \eqref{est-Delta-Pi} in the full $L^2(Q)$-norm.
%

By analogous arguments, also \eqref{est-Delta-phi} can be obtained by 
comparison from \eqref{flow-rule-pi+}.  More precisely, we might
get here for
\eqref{est-Delta-phi}  the full $L^2(Q)$-norm  upon testing 
\eqref{flow-rule-pi+}  on  $(\Delta\alpha,\Delta\phi)^\top$,
 which would be allowed if we assume to construct the finite
dimensional spaces  starting from  eigenfunctions of the Laplacian. 

 Finally,   estimate \eqref{est-of-penalization-} follows 
as the term 
$\int_0^t\int_\Omega\mathfrak{N}_\eps(\zeta_{\eps h})\DT\zeta_{\eps h}\,\d x\d t$
with $\mathfrak{N}_\eps$ from \eqref{regularization-of-N} is bounded.
\end{proof}

\begin{proposition}[Convergence of the Galerkin approximation for $h\to0$]\label{prop-conv1}
Let  assumptions  \eqref{ass} hold and let $\eps>0$ be fixed. 
Then, for $h\to0$, there  exists a not relabeled  subsequence of 
$\{(y_{\eps h},\PP_{\eps h},\alpha_{\eps h},\phi_{\eps h},\zeta_{\eps h},\por_{\eps h},\vartheta_{\eps h})\}_{h>0}^{}$
converging weakly* in the topologies indicated in \eqref{est}{\rm a-g}
to some $(y_\eps,\PP_\eps,\alpha_\eps,\phi_\eps,\zeta_\eps,\por_\eps,
\vartheta_\eps)$.
Every such limit seven-tuple is a weak solution 
to the regularized problem \eqref{system+} with the initial/boundary conditions
\eqref{BC}---\eqref{IC} modified by \eqref{BC-IC+}. Moreover, the following
a-priori estimates hold
\begin{subequations}\begin{align}\label{est++}
&\big\|{\rm div}(|\nabla\PP_{\eps}|^{q-2}\nabla\PP_{\eps})\big\|_{L^2(Q;\R^{d\times d})}\le C,
\ \   \big\|\Delta\alpha_{\eps}\big\|_{L^2(Q)}\le C,\ \  
\big\|\Delta\phi_{\eps}\big\|_{L^2(Q)}\le C,
\\&\label{est-of-penalization}
\big\|\min(0,\zeta_\eps)\big\|_{L^\infty(I;L^2(\Omega))}\le C/\sqrt\eps
\quad\text{ and }\quad
\big\|\max(1,\zeta_\eps)\big\|_{L^\infty(I;L^2(\Omega))}\le C/\sqrt\eps.
\end{align}\end{subequations}
 Furthermore,  the following strong convergences hold for $h\to0$
\begin{subequations}\label{strong-conv}
\begin{align}&&&\label{strong-conv-DT-Pi}
\DT\PP_{\eps h}^{}\PP_{\eps h}^{-1}\to\DT\PP_\eps^{}\PP_\eps^{-1}&&\text{strongly in }\ L^2(Q;\R^{d\times d}),
\\&&&\label{strong-conv-DTalpha}
\DT\alpha_{\eps h}\to\DT\alpha_\eps\ \text{ and }\ 
\DT\phi_{\eps h}\to\DT\phi_\eps\ \text{ and }\ 
\DT\zeta_{\eps h}\to\DT\zeta_\eps&&\text{strongly in }\ L^2(Q),
\\&&&\label{strong-conv-Pi-e-h}
\nabla\PP_{\eps h}\to\nabla\PP_\eps&&\text{strongly in }\ L^q(Q;\R^{d\times d\times d}),&&&&
\\&&&\label{strong-conv-mu}
\nabla\mu_{\eps h}\to\nabla\mu_\eps&&\text{strongly in }\ L^2(Q;\R^d).
\end{align}\end{subequations}
\end{proposition}

\begin{proof}
The existence of weakly* converging  not relabeled 
subsequences  follows by the classical 
 Banach selection principle.
Let us  indicate one such weak* limit by   
$(y_\eps,\PP_\eps,\alpha_\eps,\phi_\eps,\zeta_\eps,\por_\eps,
\vartheta_\eps)$  and  prove that it solves the regularized
problem \eqref{system+}.
 Note that, the 
estimates \eqref{est++}   follow  from (\ref{est}i,j) which 
are independent of $h$ and $h_0$, cf.\ \cite[Sect.~8.4]{Roub13NPDE} for this
technique. The additional estimate \eqref{est-of-penalization} is 
a consequence of 
\eqref{est-of-penalization-}.

 In order to check that weak* limits are solutions, we are called
to prove convergence of the dissipation rate term,  i.e.\ the
heat-production rate,  in the
heat-transfer equation.  This in turn requires that we prove the  strong convergence of 
$\DT\PP_{\eps h}$, $\DT\alpha_{\eps h}$, $\DT\phi_{\eps h}$, $\DT\zeta_{\eps h}$,
and of $\nabla\por_{\eps h}$, i.e.\ (\ref{strong-conv}a,b,d).   
To this  aim, let $\widetilde\PP_{h}$, $\tilde\alpha_{h}$, $\tilde\phi_{h}$, 
$\tilde\zeta_{h}$, and $\tilde\por_h$ be elements of the
finite-dimensional subspaces which are approximating $\PP_\eps$,
$\alpha_\eps$, $\phi_\eps$, $\zeta_\eps$,  and $\mu_\eps$   with respect to strong $L^2$
topologies along with the corresponding time derivatives. Such
approximants can be constructed by projections 
at the level of time derivatives.  

\def\Z{z}

We begin by discussing the terms  $\DT\alpha_{\eps h}$ and $\DT\phi_{\eps h}$, for 
they allow essentially the same  treatment. Let us introduce the shorthand 
notation 
\begin{align}\label{z-notation}
\Z:=(\alpha,\phi)
\end{align}
in this proof and, for notational simplicity, consider 
$\kappa_2=\kappa_3=:\kappa$.
We  crucially exploit  the strong monotonicity of 
 $\partial_{\DT\Z}^{}\mathfrak{D}(\alpha, \phi, \theta;\cdot)$. 
Referring to $a_\mathfrak{D}^{}>0$ from the uniform monotonicity assumption
\eqref{ass-D-monotone}, we can estimate
\begin{align}\label{large-damage-strong-conv}
&\limsup_{h\to0}\frac{a_\mathfrak{D}^{}}2\|\DT\Z_{\eps h}{-}\DT\Z_\eps\|_{L^2(Q;\R^2)}^2
\\[-.3em]&\nonumber\quad
\le\limsup_{h\to0}a_\mathfrak{D}^{}\|\DT\Z_{\eps h}{-}\DT{\tilde\Z}_{h}\|_{L^2(Q;\R^2)}^2
+\lim_{h\to0}a_\mathfrak{D}^{}\|\DT{\tilde\Z}_{h}{-}\DT\Z_\eps\|_{L^2(Q;\R^2)}^2
\\[-.3em]&\nonumber\quad
\le\limsup_{h\to0}\int_Q\big(
\partial_{\DT\Z}^{}\mathfrak{D}(\Z_{\eps h},\theta_{\eps h};\DT\Z_{\eps h})
-\partial_{\DT\Z}^{}\mathfrak{D}(\Z_{\eps h},\theta_{\eps h};\DT{\tilde\Z}_{h})
\big){\cdot}
(\DT\Z_{\eps h}-\DT{\tilde\Z}_{h})\,\d x\d t
\\[-.3em]&\nonumber\quad
=\limsup_{h\to0}
\int_Q \big(-(p_{{\rm age},\eps h},p_{{\rm eff},\eps h})-\partial_{\DT\Z}^{}\mathfrak{D}(\Z_{\eps h},\theta_{\eps h};\DT{\tilde\Z}_{h})
\big){\cdot}(\DT\Z_{\eps h}-\DT{\tilde\Z}_{h})\,\d x\d t
\\[-.3em]&\nonumber\quad
=\limsup_{h\to0}
\int_Q\REFF{\Big(}
 \partial_\alpha \FM   ( F_{{\rm el},\eps h}  ,\Z_{\eps h},
\zeta_{\eps h})(\DT{\tilde\alpha}_{h}-\DT\alpha_{\eps h})
\\[-.4em]&\nonumber\qquad\qquad
+{\rm tr}\big(\sigma'(\phi_{\eps h})\WE{\PP^{-T}_{\eps h}}\nabla y_{\eps h}^\top
\partial_\phi\FM (F_{{\rm el},\eps h},\Z_{\eps h},\zeta_{\eps h})\big)
(\DT{\tilde\phi}_{h}-\DT\phi_{\eps h})
\\[-.2em]&\nonumber\qquad\qquad
-\big(\kappa\Delta\Z_{\eps h}
+
\partial_{\DT\Z}^{}\mathfrak{D}(\Z_{\eps h},\theta_{\eps h};\DT\Z)
\big){\cdot}(\DT{\tilde\Z}_{h}-\DT\Z_{\eps h})\REFF{\Big)}\,\d x\d t
\\[-.2em]&\nonumber\quad
=\lim_{h\to0}\int_Q \REFF{\Big(}\partial_\alpha \FM  (
F_{{\rm el},\eps h} ,\Z_{\eps h},\zeta_{\eps h})
(\DT{\tilde\alpha}_{h}-\DT\alpha_{\eps h})
\\[-.4em]&\nonumber\qquad\qquad
+{\rm tr}\big(\sigma'(\phi_{\eps h}) \WE{\PP^{-T}_{\eps h}}\nabla y_{\eps
  h}^\top \partial_\phi\FM ( F_{{\rm el},\eps h} ,\Z_{\eps h},\zeta_{\eps h})\big)
(\DT{\tilde\phi}_{h}-\DT\phi_{\eps h})
\\[-.2em]&\nonumber\qquad\qquad
+
\partial_{\DT\Z}^{}\mathfrak{D}(\Z_{\eps h},\theta_{\eps h}; \DT z
){\cdot}
( \DT{\tilde\Z}_{h}-\DT\Z_{\eps h})
-\kappa\DT{\tilde\Z}_{h}\Delta\Z_{\eps h}\REFF{\Big)}\,\d x\d t
\\[-.5em]&\qquad\nonumber
+\limsup_{h\to0}\int_\Omega
\frac{\kappa}2|\nabla\Z_0|^2-\frac{\kappa}2|\nabla\Z_{\eps h}(T)|^2\,\d x
\\[-.3em]&\quad
\le
-\int_Q\kappa\DT\Z_\eps\Delta\Z_\eps\,\d x\d t+\int_\Omega
\frac{\kappa}2|\nabla\Z_0|^2-\frac{\kappa}2|\nabla\Z_\eps(T)|^2\,\d x
=0
\nonumber
\end{align}
where $\theta_{\eps h}=C_{\rm v}^{-1}(\vartheta_{\eps h})$ and 
$\theta_{\eps}=C_{\rm v}^{-1}(\vartheta_{\eps})$. In \eqref{large-damage-strong-conv},
we used 
\eqref{flow-rule+} tested by $\DT\Z_{\eps h}-\DT{\tilde\Z}_{h}$. 
Note that this is allowed at the Galerkin approximation level.   
 Note that we have used the shorthand notation 
$F_{{\rm el},\eps h} =  \PORO(\phi_{\eps h}) \nabla y_{\eps h}\,\WE{\PP^{-1}_{\eps h}}.$
 Additionally, we also   used that 
$\vartheta_{\eps h}\to\vartheta_{\eps}$ strongly in $L^2(Q)$, due to 
the Aubin-Lions theorem (in fact,  such strong convergence holds
\EE in any $L^p(Q)$ with $1\le p<10/3$ if 
$d=3$ or $1\le p<4$ if $d=2$, cf.\ \cite{Roub13NPDE}) and 
and  also that  $\theta_{\eps h}\to\theta_{\eps}$, due to the continuity of the 
superposition operator $C_{\rm v}^{-1}(\cdot)$. 
In \eqref{large-damage-strong-conv}, we also  that 
$ \partial_\alpha \FM ( F_{{\rm el},\eps h},  \alpha_{\eps h},
\phi_{\eps h},
\zeta_{\eps h})$, $ {\rm tr}\big(\sigma'(\phi_{\eps h}) \WE{\PP^{-T}_{\eps h}}\nabla y_{\eps
  h}^\top \partial_\phi\FM (F_{{\rm el},\eps h},\Z_{\eps h},\zeta_{\eps h})\big)  $, and 
$\partial_{\DT\Z}^{}\mathfrak{D}(\Z_{\eps h},\theta_{\eps h};\DT\Z)$ with fixed 
$\DT\Z$ strongly converge in $L^2(Q)$.
The last equality in \eqref{large-damage-strong-conv} relies on the
 the fact that  
$\Delta\alpha_\eps
\in L^2(Q)$ and of $\Delta\phi_\eps\in L^2(Q)$ from the estimates 
\eqref{est-Delta-phi}. In particular, the following holds 
\begin{align}\label{calculus}
\int_Q\DT\Z_\eps\Delta\Z_\eps\,\d x\d t=\int_\Omega
\frac12|\nabla\Z_0|^2-\frac12|\nabla\Z_\eps(T)|^2\,\d x,
\end{align}
cf.\ 
 the mollification-in-space arguments e.g.\ in 
\cite[Formula (3.69)]{PoRoTo10TCTF} or \cite[Formula (12.133b)]{Roub13NPDE}.
Moreover, in the last equality in \eqref{large-damage-strong-conv}
we  have used 
$\DT{\tilde\alpha}_{h}\to\DT\alpha_\eps$
and $\DT{\tilde\phi}_{h}\to\DT\phi_\eps$ strongly in $L^2(Q)$ for
$h\to0$.  This concludes the proof of the first two convergences
in  \eqref{strong-conv-DTalpha}.



As for the strong convergence \eqref{strong-conv-DT-Pi} and 
\eqref{strong-conv-Pi-e-h}, we refer to \cite{RouSte??TELS}.

The limit passage in the Galerkin approximation of 
the semilinear Cahn-Hilliard diffusion system (\ref{system+}d,e) is easy by the 
already obtained convergences. Note that, for  any  test function $v$
valued in a finite-dimensional  Galerkin  space we have  
that 
\begin{align*}
&  \int_Q\mathbb M(\Z_{\eps h})\WE{\PP^{-1}_{\eps h}}\nabla \por_{\eps
    h}{\cdot}\WE{\PP^{-1}_{\eps h}}\nabla v\,\d x\d t
\\[-.4em]&\nonumber\qquad\qquad
\to\
  \int_Q\mathbb M(\Z_{\eps})\WE{\PP^{-1}_{\eps h}}\nabla
  \por_\eps{\cdot}\WE{\PP^{-1}_{\eps h}}\nabla v\,\d x\d t.
\end{align*}
 Indeed, this follows from  $\WE{\PP^{-1}_{\eps h}}\nabla \por_{\eps h}\to
\WE{\PP^{-1}_{\eps }}\nabla \por_\eps$ weakly in $L^2(Q;\R^d)$
and $\WE{\PP^{-1}_{\eps h}}\to\WE{\PP^{-1}_{\eps}}$ strongly in $L^{2^*-\epsilon}(Q;\R^{d\times d})$
 again  by Aubin-Lions' Theorem,  for some
 $\delta>0$.   In particular,  by  testing 
the limit equation \eqref{system-Darcy+}  on  $\por_\eps$
and  adding it to   the limit equation 
\eqref{system-Cahn-Hilliard+} tested  on  $\DT\zeta_\eps$, we
exploit a cancellation of the terms $\pm \por_\eps\DT\zeta_\eps$ and obtain
\begin{align}\label{limit-Cahn-Hilliard}
\int_Q
\tauR\DT\zeta_{\eps}^2
+\partial_\zeta^{}\FM(F_{{\rm el},\eps},\Z_{\eps},\zeta_{\eps})\DT\zeta_{\eps}
+\mathbb M(\Z_{\eps})\WE{\PP^{-1}_{\eps }}\nabla \por_\eps{\cdot}\WE{\PP^{-1}_{\eps }}\nabla \por_\eps\,\d x\d t\qquad
\\[-.3em]
+\!\int_\Omega\frac{\kappa_4}2|\nabla\zeta_{\eps}(T)|^2
-\frac{\kappa_4}2|\nabla\zeta_{0}|^2\,\d x+\int_\Sigma
M(\por_\eps{-}\mu_\flat)\por_\eps 
\,\d S\d t=0.
\nonumber
\end{align}

We now can prove the strong $L^2$-convergence of 
$\DT\zeta_{\eps h}\to\DT\zeta_\eps$ and 
$\WE{\PP^{-1}_{\eps h}}\nabla \por_{\eps h}\to
\WE{\PP^{-1}_{\eps }}\nabla \por_\eps$. These two convergences 
have to be obtained simultaneously in order to   be able to
exploit a cancelation as in \eqref{C-H-tested}.
%
Using \eqref{C-H-tested}
and denoting by $a_\mathbb{M}>0$ the positive-definiteness constant of 
$\mathbb M$, we can estimate
\begin{align*}\nonumber
&\limsup_{h\to0}\int_Q\tauR\big(\DT\zeta_{\eps h}-\DT\zeta_\eps\big)^2
+a_\mathbb{M}
\big|\WE{\PP^{-1}_{\eps h}}\nabla\por_{\eps h}-\WE{\PP^{-1}_{\eps }}\nabla\por_\eps\big|^2
\WE{\,\d x\d t}
\\\nonumber
&
\le\limsup_{h\to0}\int_Q\REFF{\Big(}\tauR\big(\DT\zeta_{\eps h}-\DT\zeta_\eps\big)^2
\\\nonumber
&\ \ +\mathbb M(\Z_{\eps h})
\big(\WE{\PP^{-1}_{\eps h}}\nabla \por_{\eps h}-\WE{\PP^{-1}_{\eps }}\nabla\por_\eps\big){\cdot}\big(\WE{\PP^{-1}_{\eps h}}\nabla \por_{\eps h}-\WE{\PP^{-1}_{\eps }}\nabla\por_\eps\big)\REFF{\Big)}\,\d x\d t
\\\nonumber
&=\limsup_{h\to0}\int_Q\REFF{\Big(}\tauR\big(\DT\zeta_{\eps h}-\DT\zeta_\eps\big)^2
+\mathscr{M}(\PP_{\eps h},\Z_{\eps h})\nabla \por_{\eps h}{\cdot}\nabla \por_{\eps h}
\\\nonumber
&\ \ -2\mathbb M(\Z_{\eps h})\WE{\PP^{-1}_{\eps h}}\nabla \por_{\eps h}{\cdot}\WE{\PP^{-1}_{\eps }}\nabla \por_\eps
+\mathbb M(\Z_{\eps h})\WE{\PP^{-1}_{\eps }}\nabla \por_\eps{\cdot}\WE{\PP^{-1}_{\eps }}\nabla \por_\eps\REFF{\Big)}\,\d x\WE{\d t}
\\\nonumber
&
\UU{\le}{by \eqref{C-H-tested}}
\lim_{h\to0}\bigg(\int_Q
\REFF{\Big(}\tauR\DT\zeta_\eps^2-2\tauR\DT\zeta_{\eps h}\DT\zeta_\eps
-\partial_\zeta^{}\FM(F_{{\rm el},\eps h},\Z_{\eps h},\zeta_{\eps h})\DT\zeta_{\eps h}
\\\nonumber
&\ \ -2\mathbb M(\Z_{\eps h})\WE{\PP^{-1}_{\eps h}}\nabla \por_{\eps h}{\cdot}\WE{\PP^{-1}_{\eps }}\nabla \por_\eps
+\mathbb M(\Z_{\eps h})\WE{\PP^{-1}_{\eps }}\nabla \por_\eps{\cdot}\WE{\PP^{-1}_{\eps }}\nabla \por_\eps\REFF{\Big)}\,\d x\d t
\\\nonumber
&\ \ +\int_\Sigma
M\mu_\flat\por_{\eps h}\,\d S\d t
+\int_\Omega\frac{\kappa_4}2|\nabla\zeta_{0 h}|^2\,\d x\bigg)
\\\nonumber
&\ \ -\liminf_{h\to0}\bigg(\int_\Sigma M\por_{\eps h}^2\,\d S\d t+\int_\Omega\frac{\kappa_4}2|\nabla\zeta_{\eps h}(T)|^2\,\d x\bigg)
\\\nonumber
&\le\int_Q
-\tauR\DT\zeta_{\eps}^2
-\partial_\zeta^{}\FM(F_{{\rm el},\eps},\Z_{\eps},\zeta_{\eps})\DT\zeta_{\eps}
-\mathbb M(\Z_{\eps})\WE{\PP^{-1}_{\eps }}\nabla
\por_\eps{\cdot}\WE{\PP^{-1}_{\eps }}\nabla \por_\eps\,\d x\d t 
\\[-1em]
&\qquad\qquad\qquad 
-\int_\Sigma
M(\por_\eps{-}\mu_\flat)\por_\eps \,\d S\d t
+\int_\Omega\frac{\kappa_4}2|\nabla\zeta_{0}|^2
-\frac{\kappa_4}2|\nabla\zeta_{\eps}(T)|^2\,\d x\UU{=}{by \eqref{limit-Cahn-Hilliard}}0.
\end{align*}
This entails the strong convergence for  $\DT\zeta_{\eps h}$ from 
\eqref{strong-conv-DTalpha}  as well as that of 
terms  $\WE{\PP^{-1}_{\eps h}}\nabla \por_{\eps h}$.
From this, we  obtain  the strong convergence
\eqref{strong-conv-mu}  for  
$\nabla \por_{\eps h}\to\nabla \por_\eps$.  Note however that this last
convergence is not exploited in the following. 


The convergence of the mechanical part for $h\to0$ is  now
straightforward. As   highest-order terms in (\ref{system+}a,c-e)
are linear,  weak convergence together and Aubin-Lions compactness for
lower-order terms  suffices. The limit passage in
the quasilinear $q$-Laplacian in (\ref{system+}b) as well as 
in the $\mathfrak{R}$- and $\mathfrak{D}$-terms  in  (\ref{system+}b,c) 
follows from the already proved strong convergences (\ref{strong-conv}c)
and (\ref{strong-conv}a,b), respectively. 

Eventually, the limit passage in the semilinear heat-transfer equation 
\eqref{system-heat+}  can be ascertained due \EE
to the already proved strong convergences  (\ref{strong-conv}a,b,d),
 allowing indeed the passage to the limit   in the (regularized) right-hand side. 
\end{proof}

 In order to remove the regularization by passing to the   limit for
$\eps\to0$, we cannot  directly  rely on  estimates
\eqref{est-theta-H1}-\eqref{est-w} and \eqref{est-of-penalization-},
for these are depending  $\eps>0$.
On the other hand,
having already  passed to the limit in $h$ we are now in the
position of performing a number of nonlinear tests for the heat equation, which are
specifically tailored to the $L^1$-theory. 

\begin{lemma}[Further a-priori estimates for
  temperature]\label{lem-est+}
 Let $\vartheta_\eps$ be the (rescaled) temperature component of
the weak solution to the regularized problem \eqref{system+} whose
existence is proved  in 
Proposition~{\rm \ref{prop-conv1}}.  Then, 
\begin{equation}
  \label{positivity}
  \vartheta_\eps \geq 0 \quad \text{a.e. in} \ Q.
\end{equation}
 Moreover, one has that 
\begin{subequations}\label{est+}\begin{align}\label{est-theta-L1}
&\exists C_1>0: \quad \|\vartheta_\eps\|_{L^\infty(I;L^1(\Omega))}\le C_1,
\\&
\forall 1\le s<(d{+}2)/(d{+}1) \ \exists C_s >0 : \quad
    \|\nabla\vartheta_\eps\|_{L^s(Q;\R^d)}\le C_s
\label{est-theta}
\end{align}\end{subequations}
where the constants $C_1, \, C_s$ are independent of $\eps$.
\end{lemma}




\begin{proof}
See \cite{RouSte??TELS}.
\end{proof}

\begin{proposition}[Convergence of the regularization for $\eps\to0$]\label{prop-conv2}
 Under assumptions \eqref{ass}, as  $\eps\to0$ there 
exists a    subsequence of
$\{(y_{\eps},\PP_{\eps},\alpha_{\eps},\phi_{\eps},\zeta_{\eps},\por_{\eps},\vartheta_{\eps})\}_{\eps>0}^{}$
(not relabeled)
which converges weakly* in the topologies indicated in 
{\rm(\ref{est}a-f)}, \eqref{est++}, and 
\eqref{est+} to some $(y,\PP,\alpha,\phi,\zeta,\por,\vartheta)$.
Every such a limit seven-tuple is a weak solution to the original problem 
in the sense of Definition~{\rm \ref{def}}. Moreover, the following 
strong convergences hold
\begin{subequations}\label{strong-conv+}
\begin{align}&&&\label{strong-conv-DT-Pi+}
\DT\PP_{\eps}^{}\PP_{\eps}^{-1}\to\DT\PP\PP^{-1}&&\text{strongly in }\ L^2(Q;\R^{d\times d}),
\\&&&\label{strong-conv-DTalpha+}
\DT\alpha_{\eps}\to\DT\alpha\ \text{ and }\ 
\DT\phi_{\eps}\to\DT\phi\ \text{ and }\ 
\DT\zeta_{\eps}\to\DT\zeta&&\text{strongly in }\ L^2(Q),
\\&&&\label{strong-conv-Pi-e-h+}
\nabla\PP_{\eps}\to\nabla\PP&&\text{strongly in }\ L^q(Q;\R^{d\times d\times d}),&&&&
\\&&&\label{strong-conv-mu+}
\nabla\mu_{\eps}\to\nabla\mu&&\text{strongly in }\ L^2(Q;\R^d).
\end{align}\end{subequations}
 Eventually,  the regularity \eqref{Laplaceans} and 
the energy conservation \eqref{energy-conserv++} hold.
\end{proposition}

\begin{proof}
Again, by the Banach selection principle, we  can extract a not
relabeled   weakly* convergent
subsequence  with respect to the topologies in 
{\rm(\ref{est}a-f)}, \eqref{est++}, and 
\eqref{est+} and  indicate its limit by  $(y,\PP,\alpha,\phi,\zeta,\por,\vartheta)$.
\def\Z{z}

The improved, strong convergences \eqref{strong-conv+}  can be
obtained by arguing as in 
the proof of \eqref{strong-conv} in Proposition~\ref{prop-conv1},
cf.\ \cite{RouSte??TELS} for details as far as the plastic strain concerns. 

 The passage to the limit into the various relations follows
similarly as in the proof of Proposition \ref{prop-conv1}. Instead
of repeating the whole argument, we limit ourselves in pointing out
the few differences. 

 A first difference concerns the  limit passage towards the 
inclusion governing $\mu$, due to the  presence of the  constraints $0\le\zeta\le1$.  
From \eqref{est-of-penalization}  one has that  $\zeta$
is valued in $[0,1]$. To facilitate the limit passage towards the variational 
inequality \eqref{mu-weak}, we write \eqref{system-Cahn-Hilliard+} in the form
\begin{align*}\nonumber
&\int_Q\FM( F_{{\rm el},\eps},  \alpha_\eps,\phi_\eps,\tilde\zeta)
-\FM( F_{{\rm el},\eps}, \alpha_\eps,\phi_\eps,\zeta_\eps)-\por(\tilde\zeta-\zeta_\eps)
-\kappa_4\nabla\zeta_\eps{\cdot}\nabla(\tilde\zeta-\zeta_\eps)
\\[-.3em]&\nonumber\qquad
+\tauR\DT\zeta_\eps\tilde\zeta+\widehat{\mathfrak{N}}_\eps(\tilde\zeta)\,\d x\d t
+\int_\Omega\frac12\tauR\zeta^2_0\,\d x
\ge\int_Q\widehat{\mathfrak{N}}_\eps(\zeta_\eps)\,\d x\d t
+\int_\Omega\frac12\tauR\zeta_\eps^2(T)\,\d x
\end{align*} 
 where  $\widehat{\mathfrak{N}}_\eps$  is the primitive  of 
$\mathfrak{N}_\eps$ from \eqref{regularization-of-N} with 
$\mathfrak{N}_\eps(0)=0$. For  all  $\tilde\zeta$ valued in
$[0,1]$, see
\eqref{mu-weak}, the term $\widehat{\mathfrak{N}}_\eps(\tilde\zeta)$
vanishes.  The limit passage hence ensues by classical  
continuity or lower semicontinuity arguments. 

The strong convergence of $\vartheta_\eps$  follows  again by the Aubin-Lions 
Theorem.  Nevertheless, we use here a  
coarser topology  with respect to  than in
Proposition~\ref{prop-conv1}.  This change is however immaterial
with respect to  the limit passage in the mechanical
part (\ref{system}a-e). Actually, some arguments are even 
simplified, for  we do not need to approximate the limit into the 
finite-dimensional 
subspaces as we  did  in \eqref{large-damage-strong-conv}.
%
%
The heat-production rate $r$ on the right-hand side
of \eqref{system-heat+} converges now strongly in $L^1(Q)$.

 Eventually, the regularity \eqref{Laplaceans} can be obtained from the estimates 
\eqref{est++}, which are uniform in $\eps>0$.
%
%
The energy conservation \eqref{energy-conserv++} follows directly from the 
energy conservation in the mechanical part, as essentially used 
above while checking  the strong convergences  \eqref{strong-conv+}. 
Indeed, one integrates \eqref{energy-conserv+} over $[0,t]$ and
sum it to the heat equation tested on the constant 1. Note that this
is amenable as the constant $1$ can be put in  
duality with $\DT\vartheta$, so that the chain-rule applies.  
\end{proof}

\REFF{
\section{Conclusion}\label{sec-concl}
We have addressed a model used in geophysics for poroelastic damagable rocks
with plastic-like strain, which can accomodate the large displacement
occuring during long geological time scales. The model is anisothermal,
so e.g.\ effects of flash heating on tectonic faults during ongoing earthquakes 
can be captured in this model; this may immitate a popular 
Dieterich-Ruina rate-and-state friction model  \cite{Diet79MRF,Ruin83SISVFL} 
which otherwise does not seem to allow for a rational thermodynamical formulation, 
cf.\ \cite{Roub14NRSD} for this interpretation at small-strain context.

Also inertia is considered, so 
that seismic waves emitted during tectonic earthquakes in the Earth crust 
(i.e.\ the solid, very upper part of the mantle) can be captured in the model.  
The model is formulated at large strains and complies with  frame
indifference. The main assumption of the model is that the elastic
Green-Lagrange strain is small. In contrast with the small-strain
but large-displacement model in \cite{LyHaBZ11NLVE},
where the energy does not seem to be completely conserved no matter 
how the Korteweg-like stress (usually balacing the energy) is devised,
cf.\ \cite{Roub17GMHF}, the present large-strain model is thermodynamically
consistent. 
At the same time, our assumption about the smallness of 
elastic Green-Lagrange strains and the nearly isochoric nature of  
plastification is not in direct conflict with the mentioned geophysical 
applications.

 
The smallness assumption on the elastic 
Green-Lagrange strains could be avoided by suitably modifying
relations and the analytical treatment in the existence proof. Namely, one could consider a nonlocal 
nonsimple gradient theory for the total strain, which would allow to
control the displacement 
in the Sobolev-Slobodetski\u\i\ Hilbert space $W^{2+\gamma,2}(\Omega;\R^d)$ with 
$\gamma>d/2$. Then, one could use the Healey-Kr\"omer theorem twice, both for 
$F_{\rm el}$ and for $\PP$, provided $\FM$ has sufficiently fast growth 
for $\det F_{\rm el}\to~0+$, cf.\
\cite[Sect.\,9.4.3]{KruRou18MMCM}. Such 
nonlocal models allow for the description of more general dispersion phenomena, as shown in 
\cite{Jira04NTCM}. Another relevant option could be that of
considering a gradient 
theory for $F_{\rm el}$ rather than for $F$ but, at this moment, this
seems to pose analytical difficulties.  

This also opens a possibility of avoiding Dirichlet boundary condition for 
$\PP$ completely,
as already mentioned \cite[Remark 4.5]{RouSte??TELS} but rather for the case of 
full hardening only. Here, 
controlling the inverse the elastic strain $F_{\rm el}$ as outlined above 
would allow us to estimate $\PP = F_{\rm el}^{-1} \nabla y S^{-1}$. Thus such approach would be 
amenable even in the most natural case of homogeneous Neumann boundary conditions 
for $\PP$ on the whole boundary $\partial\Omega$.
}

\section*{Acknowledgments}
The authors are very thankful to Alexander Mielke for inspiring 
conversations about the model and to Vladimir Lyakhovsky for many discussions 
during previous years as well as comments to a working version of this
paper. 

\end{document}